\documentclass[11pt,a4paper,reqno]{amsart}

\usepackage[a4paper,lmargin=2.5cm,rmargin=2.5cm,tmargin=3cm,bmargin=3cm]{geometry}

\usepackage[english]{babel}

\usepackage{amsmath,amssymb,amsthm}
\usepackage{hyperref}

\usepackage{dot2texi}
\usepackage{tikz}
\usepackage{float}
\usepackage{pgfplots}
\usepackage{csquotes}
\usetikzlibrary{calc,arrows, arrows.meta}

\usepackage{centernot}

\newtheorem*{theorem*A}{Theorem A}
\newtheorem*{theorem*B}{Theorem B}
\newtheorem*{theorem*C}{Theorem C}
\newtheorem*{theorem*D}{Theorem D}
\newtheorem{theorem}{Theorem}[section]
\newtheorem{lemma}[theorem]{Lemma}

\newtheorem{proposition}[theorem]{Proposition}
\newtheorem{corollary}[theorem]{Corollary}
\theoremstyle{definition}
\newtheorem{definition}[theorem]{Definition}
\newtheorem{example}[theorem]{Example}

\newtheorem{question}[theorem]{Question}
\theoremstyle{remark}
\newtheorem{remark}[theorem]{Remark}
\numberwithin{equation}{section}

\newcommand{\vertiii}[1]{{\left\vert\kern-0.25ex\left\vert\kern-0.25ex\left\vert #1 
    \right\vert\kern-0.25ex\right\vert\kern-0.25ex\right\vert}}

\def\fnote#1{\footnote}

\def\ignora#1{}
\def\n3#1{\left\vert  \! \left\vert \! \left\vert \, #1 \, \right\vert \!
	\right\vert \! \right\vert }

\DeclareMathOperator{\dist}{dist\,}

\DeclareMathOperator{\conv}{conv}

\DeclareMathOperator{\bc}{bc}
\newcommand{\cconv}{\overline{\conv}}

\renewcommand{\geq}{\geqslant}
\renewcommand{\leq}{\leqslant}

\newcommand{\Free}{{\mathcal F}}

\newcommand{\Lip}{{\mathrm{Lip}}_0}

\newcommand{\SE}{\operatorname{StrExp}}

\newcommand{\aconv}{\overline{\mathop\mathrm{aconv}}}

\newcommand{\Mol}[1]{\operatorname{Mol}\left(#1\right)}

\newcommand{\pten}{\ensuremath{\widehat{\otimes}_\pi}}

\begin{document}
	
	\title{ Dual Banach spaces with the ball-covering property }

    \author[Langemets]{Johann Langemets}
\address[Langemets]{Institute of Mathematics and Statistics, University of Tartu, Narva mnt 18, 51009 Tartu, Estonia}
\email{johann.langemets@ut.ee}
\urladdr{\url{https://www.johannlangemets.com/}}
\urladdr{
\href{https://orcid.org/0000-0001-9649-7282}{ORCID: \texttt{0000-0001-9649-7282} } }

\author[Mõttus]{Emma Mõttus}
\address[Mõttus]{Institute of Mathematics and Statistics, University of Tartu, Narva mnt 18, 51009 Tartu, Estonia}
\email{emma.mottus@ut.ee}

\author[Saealle]{Natalia Saealle}
\address[Saealle]{Institute of Mathematics and Statistics, University of Tartu, Narva mnt 18, 51009 Tartu, Estonia}
\email{natalia.saealle@ut.ee}
\urladdr{\href{https://orcid.org/0009-0008-1511-228X}{ORCID: \texttt{0009-0008-1511-228X}}}


	\thanks{This work was supported by the Estonian Research Council grant (PRG2545).}


	\begin{abstract}
		We study ball-covering properties of dual Banach spaces and their connections with the geometry of predual unit balls. One of our main results shows that, for every separable Banach space \(X\), the unit ball \(B_X\) is a slicely countably determined set if and only if \(\bc(X^*)=1\), where $\bc(\cdot)$ is the ball-covering index introduced by A.~J.~Guirao, A.~Lissitsin, and V.~Montesinos. We obtain several sufficient conditions for the uniform ball-covering property in dual spaces, including duals of spaces with a \(K\)-unconditional basis for \(K<2\), and duals of separable spaces whose unit ball is the closed convex hull of a set of uniformly strongly exposed points. The constant \(2\) is sharp: there is a space with a \(2\)-unconditional basis whose dual fails the ball-covering property. Applications are given to spaces of operators and to Lipschitz spaces. In particular, \(\mathcal L(L_p[0,1])\) has the uniform ball-covering property for every \(1<p<\infty\), which answers a question posed by Q.~Bao, R.~Liu, and J.~Shen. As an application to Lipschitz spaces, we prove that $\Lip(M)$ has the uniform ball-covering property whenever $M$ is a separable complete ultrametric or Hölder metric space.
	\end{abstract}
	
	\maketitle
	
	\section{Introduction}\label{sec:introduction}

A Banach space \(X\) is said to have the ball-covering property
(BCP) if its unit sphere can be covered by countably many open balls
that do not contain the origin. Equivalently, there are points
\(x_n\in X\) and positive numbers \(r_n\) such that
\begin{equation}\label{eq: BCP def}
        S_X\subset \bigcup_{n=1}^{\infty} B(x_n,r_n)
        \qquad\text{and}\qquad
        \|x_n\|\ge r_n \quad(n\in\mathbb N).
\end{equation}
This covering property was introduced by L.~Cheng in \cite{Cheng2006}. Several quantitative refinements of the BCP are now available. Z.~Luo and
B.~Zheng introduced in \cite{LuoZheng2021} the \emph{strong ball-covering property} (SBCP),
where the radii of the covering balls are required to be uniformly
bounded, and the \emph{uniform ball-covering property} (UBCP), where the covering
balls are required, in addition, to stay away from a fixed positive
multiple of the unit ball.

From the definitions, it is straightforward to see that every separable Banach space has the UBCP. Note that $\ell_\infty$ is a prototypical example of a nonseparable Banach space with the UBCP \cite{Cheng2006}. It is known that Banach spaces with the BCP have weak$^*$-separable duals \cite{Cheng2006}. Moreover, V.~Fonf and C.~Zanco proved in \cite{FZ2009} that $X$ has a weak$^*$-separable dual if and only if for every $\varepsilon>0$ there is a $(1+\varepsilon)$-equivalent norm $|\cdot|$ on $X$ so that $(X,|\cdot|)$ has the SBCP.

It was shown in \cite{CCL2008} that the BCP is not preserved under linear isomorphisms by equivalently renorming $\ell_\infty$. An easier example is to consider the isomorphic spaces $\ell_\infty$ and $L_\infty[0,1]$ \cite{LZ2020}. In \cite{LLLZ2022}, the authors proved that the BCP does not pass even to 1-complemented subspaces. The BCP is closely related to several important properties, such as the Radon--Nikod\'ym property, locally uniform rotund spaces, smoothness, and countable $\pi$-bases; see, for instance,
\cite{Cheng2011,CCL2008, CKWZ2010, GLM2019, HKL2024,LLLZ2022,Sain2025}.

Let $\alpha\in[-1,1)$. Another quantitative approach, due to A.~Guirao, A.~Lissitsin and V.~Montesinos
\cite{GLM2019}, is the \emph{strictly \(\alpha\)-off ball-covering property}
(\(\alpha\)-BCP), where the norms of the centers of the balls in \eqref{eq: BCP def} are required to satisfy \(\|x_n\|\ge r_n+\alpha\). The associated index
\[
        \bc(X):=\sup\{\alpha\in[-1,1): X\text{ has the }\alpha\text{-BCP}\}
\]
measures how far the covering balls may be forced away from the origin.
Thus \(\bc(X)=1\) is the optimal possible form of the BCP in this scale. All of the above-mentioned ball-covering properties are related as depicted in Figure~\ref{fig:relations-between-ball-covering-properties}.

The main purpose of this paper is to study these properties for dual
Banach spaces. First let us note that for a dual space to have the BCP it is necessary that the predual is a separable Banach space \cite[Remark~11.2]{GLM2019}. The guiding question is the following: 
to what extent can
a ball-covering property of \(X^*\) be read from the geometry of the
predual unit ball \(B_X\)? A first important result in this direction was
obtained in \cite{GLM2019}: if \(X\) is separable and \(B_X\) is the closed
convex hull of its strongly exposed points, then \(\bc(X^*)=1\). Our first
main theorem gives the exact geometric condition behind this phenomenon.

Recall that the notion of \emph{slicely countably determined} (SCD) sets (see Definition~\ref{def: SCD set}) was
introduced by A.~Avil\'es, V.~Kadets, M.~Martín, J.~Mer\'i, and V.~Shepelska in \cite{AKMMS2010}
and includes, in particular, many classical situations coming from the
Radon--Nikod\'ym property, Asplund spaces, and spaces with 1-unconditional bases.

Our first main result (see Theorem~\ref{thm: BCP in dual}) characterizes the optimal ball-covering index in dual spaces.
\begin{theorem*A}
     Let $X$ be a separable Banach space. Then $B_X$ is an SCD set if and only if $\bc(X^*)=1$.
\end{theorem*A}
This result extends the strongly exposed point criterion of A.~Guirao, A.~Lissitsin and V.~Montesinos
\cite{GLM2019} to a full characterization and connects the ball-covering properties to the theory of SCD sets.
Moreover, we prove that if \(X\) is separable and \(X^*\) fails the \((-1)\)-BCP, then there is an equivalent dual norm on \(X^*\) for which the resulting dual space has the BCP, but the corresponding renormed predual unit ball is not an SCD set (see Proposition~\ref{prop:dual-bcp-but-bc-less-than-one}).  

Our second main result is about the UBCP of dual spaces (see Proposition~\ref{prop:k-unconditional-dual-bcp} and Theorem~\ref{thm: UBCP in dual}).
\begin{theorem*B}
    Let $X$ be a separable Banach space. Assume that either of the following conditions holds:
\begin{itemize}
    \item $X$ has a $K$-unconditional basis, where $1\leq K<2$;
    \item \(B_X\) is the closed convex hull of a set of uniformly strongly exposed points.
\end{itemize}
Then \(X^*\) has the UBCP.
\end{theorem*B}
Our first source of examples of dual spaces with the UBCP comes from
unconditional bases.
 V.~Kadets, M.~Mart\'in, J.~Mer\'i, and D.~Werner proved that the unit ball is an SCD set for every Banach space with a 1-unconditional basis \cite[Theorem~3.1]{KMMW2013}. Recently, M.~Lõo and Y.~Perreau constructed
a Banach space with a 2-unconditional basis whose unit ball is not SCD
\cite{LP2026}. It is an open question whether the unit ball of every Banach space with a $K$-unconditional basis, where $1<K<2$, is an SCD set \cite{LP2026}. Note that our dual uniform covering result is independent of the SCD problem. Nevertheless, using the space in \cite{LP2026}, we prove that the dual of this space even fails the BCP (see Proposition~\ref{prop:two-unconditional-dual-no-bcp}), hence the constant 2 is optimal.

Our second source of UBCP examples is based on uniformly strongly exposed
sets (see Definition~\ref{def: USE set}). J.~Lindenstrauss introduced this notion in his work on norm-attaining
operators \cite{L1963}, where he showed that if \(B_X\) is the closed
convex hull of a set of uniformly strongly exposed points, then \(X\) has
what is now called Lindenstrauss property A. We prove that the same
geometric assumption implies that \(X^*\) has the UBCP and \(\bc(X^*)=1\) (see Theorem~\ref{thm: UBCP in dual}).
However, Lindenstrauss property A alone does not imply any strong form of bounded
ball-covering in the dual. We construct a separable Banach space \(X\)
with Lindenstrauss property A such that \(\bc(X^*)=1\), but \(X^*\) fails
the SBCP (see Proposition~\ref{prop: property A does not imply SBCP}). 

The abstract dual criteria above have two main applications. The first concerns
spaces of operators. The ball-covering property in spaces of operators has been recently studied in several papers \cite{BLSpreprint, BLS2025,LLLZ2022}. In particular, it has been proven that $\mathcal{L}(c_0)$ and $\mathcal{L}(\ell_p)$, $1\leq p<\infty$, have the UBCP \cite[Corollaries~3.1 and 3.4]{LLLZ2022}, and that $\mathcal{L}(L_p[0,1])$ has the UBCP whenever $3/2<p<3$ \cite[Theorem~3.3]{BLS2025}. In \cite[Question~3]{BLS2025}, it is asked whether $\mathcal{L}(L_p[0,1])$ has the BCP whenever $1<p<\infty$. Our main application is the following theorem (see Theorem~\ref{thm: UBCP in L(X,Y*)}), which answers the aforementioned question positively and provides an alternative way to simultaneously prove that $\mathcal{L}(\ell_p)$ and $\mathcal{L}(L_p[0,1])$ have the UBCP whenever $1<p<\infty$.
\begin{theorem*C}
    Let $X$ and $Y$ be separable Banach spaces such that $B_X$ and $B_Y$ are the closed convex hulls of sets of uniformly strongly exposed points. Then $\mathcal{L}(X,Y^*)$ and $\mathcal{L}(Y,X^*)$ both have the UBCP, and $\bc (\mathcal{L}(X,Y^*))=\bc(\mathcal{L}(Y,X^*))=1$. 
\end{theorem*C}

The second class of applications concerns Lipschitz spaces. If \(M\) is
a pointed complete metric space, then \(\operatorname{Lip}_0(M)\) is the
dual of the Lipschitz-free space \(\mathcal F(M)\). Hence the preceding
duality results naturally apply to \(\operatorname{Lip}_0(M)\).  The known
picture for infinite metric spaces is quite limited. On the positive
side, if \(M\) is separable and complete and \(\mathcal F(M)\) has the
Radon--Nikodým property, then \(\bc(\operatorname{Lip}_0(M))=1\); in
particular this applies to complete purely \(1\)-unrectifiable metric
spaces by \cite{AGPP2022}.  On the negative side, if \(M\) is a length
metric space, then \(\operatorname{Lip}_0(M)\) has the Daugavet property
\cite{GLPRZ2018} and fails the \((-1)\)-BCP \cite{CLL2023}.

We add two new types of examples. First, for the unit circle \(\mathbb T\)
with the chordal metric, we prove that \(\operatorname{Lip}_0(\mathbb T)\)
has the BCP, indeed \(bc(\operatorname{Lip}_0(\mathbb T))=1\), but fails
the UBCP (see Theorem~\ref{thm:lip-torus-no-ubcp}). This gives a natural Lipschitz-space separation between the BCP
and the UBCP. Secondly, in Theorem~\ref{thm: UBCP in Lip}, we prove the following result. 

\begin{theorem*D}
    If $M$ is a separable complete uniformly Gromov concave metric space, then $\Lip(M)$ has the UBCP and $\bc(\Lip(M))=1$.
\end{theorem*D}

This theorem applies, in
particular, to separable ultrametric spaces and to Hölder metric spaces \cite{CM2019},
using the known characterization of uniformly Gromov concavity in terms
of uniformly strongly exposed molecules in \(\mathcal F(M)\)
\cite{CCGLMRZ2019}.

The paper is organized as follows. Section~\ref{sec: preliminaries} contains notation and
preliminary material on the ball-covering properties and related notions.
In Section~\ref{sec: main results}, we prove the main dual results: the SCD characterization of
\(bc(X^*)=1\) (see Theorem~\ref{thm: BCP in dual}), the \(K<2\) unconditional basis criterion (see Proposition~\ref{prop:k-unconditional-dual-bcp}), the sharp
\(2\)-unconditional counterexample (see Proposition~\ref{prop:two-unconditional-dual-no-bcp}), and the uniformly strongly exposed
criterion for the UBCP (see Theorem~\ref{thm: UBCP in dual}). Section~\ref{sec: L(X,Y)} applies these results to spaces of
operators via projective tensor products, and the main results of this section are Theorems~\ref{thm: BCP in L(X,Y)} and \ref{thm: UBCP in L(X,Y*)}. Section~\ref{sec: Lipschitz spaces} is devoted to
Lipschitz spaces, including the example of the circle (see Theorem~\ref{thm:lip-torus-no-ubcp}), and the positive
results for metric spaces for which $B_{\mathcal F(M)}$ is the closed convex hull of its denting points (see Proposition~\ref{prop: BCP in Lipschitz}) and uniformly Gromov concave metric spaces (see Theorem~\ref{thm: UBCP in Lip}).

\section{Notation and preliminary results}\label{sec: preliminaries}
The aim of this section is to fix notation and to record the elementary
facts on ball-covering properties and geometric notions of unit balls
which will be used throughout the paper. In particular, we collect the
basic implications between the quantitative ball-covering properties,
the terminology concerning SCD sets, and the notions of strong and
uniform strong exposure.

\subsection{General notation}

Throughout the paper, all Banach spaces are assumed to be real.
We use standard notation as in \cite{FHHMZ2011}. Our notation is completely standard and follows standard texts as \cite{FHHMZ2011}. For a Banach space \(X\), we denote by \(B_X\) and \(S_X\) its closed unit ball and unit sphere, respectively. 
The open ball with center $x$ and radius $r>0$ is denoted by \(B(x,r)\). 
For a subset \(A\) of a Banach space, we write \(\cconv(A)\) and \(\aconv(A)\) for the closed convex hull and the closed absolutely convex hull of \(A\), respectively. 

Let \(X\) be a Banach space with a Schauder basis \((e_n)\).
Recall that \((e_n)\) is called \emph{unconditional} if the basis
expansion
\(
x=\sum_{n\ge 1}^{\infty} a_n e_n
\)
of every \(x\in X\) converges unconditionally. Equivalently, there is
a constant \(K\geq 1\) such that
\begin{equation}\label{eq: uncond basis}
    \left\|\sum_{n\geq 1}a_ne_n\right\|
    \leq
    K\left\|\sum_{n\geq 1}b_ne_n\right\|
\end{equation}
whenever \(a=(a_n),b=(b_n)\in c_{00}\) and
\(|a_n|\leq |b_n|\) for every \(n\in\mathbb N\)
\cite{AK2006}.

The unconditional basis constant \(K_u\) of \((e_n)\) is the least
constant \(K\) for which inequality \eqref{eq: uncond basis} holds.
We say that \((e_n)\) is \emph{\(K\)-unconditional} whenever \(K\geq K_u\).

Let $X$ and $Y$ be Banach spaces. By $\mathcal{L}(X,Y)$ we will denote the space of bounded linear operators from $X$ to $Y$.

The projective tensor product of $X$ and $Y$, denoted by $X \pten Y$, is the completion of the algebraic tensor product $X \otimes Y$ endowed with the norm
$$
\|z\|_{\pi} := \inf \left\{ \sum_{n=1}^k \|x_n\| \|y_n\|\colon  z = \sum_{n=1}^k x_n \otimes y_n \right\},$$
where the infimum is taken over all such representations of $z$ (see \cite{R2002}). It is well known that $\|x \otimes y\|_{\pi} = \|x\| \|y\|$ for every $x \in X$, $y \in Y$, and that the closed unit ball of $X \pten Y$ is the closed convex hull of the set $B_X \otimes B_Y = \{ x \otimes y\colon  x \in B_X, y \in B_Y \}$. 
Recall that, for Banach spaces \(X\) and \(Y\), there are
canonical isometric identifications
\[
        \mathcal L(X,Y^*)
        \cong
        (X\widehat\otimes_\pi Y)^*
        \cong
        \mathcal L(Y,X^*).
\]
For more background on tensor products and the identifications above, we refer the reader to \cite{R2002}.

Given a sequence $\{X_n\colon n\in \mathbb{N}\}$ of Banach spaces, we write
$$
\left(\bigoplus_{n=1}^\infty X_n\right)_{\ell_1},\qquad \left(\bigoplus_{n=1}^\infty X_n\right)_{\ell_\infty}
$$
to denote, respectively, the $\ell_1$-sum and the $\ell_\infty$-sum of the sequence. 

\subsection{Slices and SCD sets}

If $C$ is a bounded convex subset of a Banach space $X$, by a \textit{slice} of $C$ we will mean the non-empty intersection of an open half-space with the set $C$. Every slice can be written in the form 
$$S(C,f,\delta):=\{x\in C\colon  f(x)>\sup f(C)-\delta\}$$
for suitable $f\in X^*$ and $\delta>0$. 

The notion of a slicely countably determined set was first introduced by A.~Avil\'es,  V.~Kadets, M.~Mart\'in, J.~Mer\'i, and V.~Shepelska in \cite[Definition~2.5]{AKMMS2010}. A separable Banach space $X$ is said to be slicely countably determined if every convex bounded subset of $X$ is an SCD set. These spaces generalize simultaneously Banach spaces which are Asplund or have the Radon--Nikod\'ym property.

\begin{definition}[{\cite{AKMMS2010}}]\label{def: SCD set}
    Let $X$ be a Banach space. A convex bounded subset $A$ of $X$  is said to be a \emph{slicely countably determined set} (for short, SCD set) if there is a sequence of slices $\{S_n\colon n\in\mathbb N\}$ of $A$ such that for every sequence $\{b_n\}_{n\in\mathbb N}$ with $b_n\in S_n$, one has that $A\subseteq \overline{\conv}(\{b_n\colon n\in\mathbb N\})$. 
\end{definition}
Definition~\ref{def: SCD set} is equivalent to the following: there is a sequence of slices $\{S_n\colon n\in\mathbb N\}$ of $A$ such that for every slice $S$ of $A$, there is some $n\in\mathbb N$ such that $S_n\subseteq S$ \cite{AKMMS2010}. The unit ball of a separable Banach space $X$ is SCD, for example, in the following cases:
\begin{itemize}
    \item $X$ is Asplund;
    \item $X$ has the Radon--Nikod\'ym property;
    \item $X$ has a 1-unconditional basis \cite[Theorem~3.1]{KMMW2013};
    \item $X$ is a separable Banach space with a locally uniformly rotund norm \cite[Example~2.10]{AKMMS2010}.
\end{itemize}

\subsection{Denting and strongly exposed points}
\begin{definition}
Let \(C\) be a bounded closed convex subset of a Banach space \(X\). 
\begin{itemize}
    \item A point \(x\in C\) is called a \emph{denting point} of \(C\) if for every
\(\varepsilon>0\) there is a slice \(S\) of \(C\) such that
\(x\in S\) and \(\operatorname{diam}(S)<\varepsilon\).
\item  A point \(x\in C\) is called a \emph{strongly exposed point} of \(C\) if there is
\(f\in X^*\) such that \(f(x)=\sup f(C)\) and, for every \(\varepsilon>0\),
there exists \(\delta>0\) such that
\[
        y\in C,\quad f(y)>\sup f(C)-\delta
        \quad\Longrightarrow\quad
        \|y-x\|<\varepsilon .
\]
In this case we say that \(f\) strongly exposes \(C\) at \(x\).
\end{itemize}

\end{definition}

The following concept of a uniformly strongly exposed set was first introduced by J. Lindenstrauss in \cite{L1963}, where he used it to study the density of norm-attaining operators.

\begin{definition}\label{def: USE set}
Let \(X\) be a Banach space. A subset \(U\subset S_X\) is said to be a \emph{set of uniformly strongly exposed points} if there is a family of functionals \(\{f_u\}_{u\in U}\) with
\[
\|f_u\|=f_u(u)=1
\qquad\text{for every }
u\in U,
\] such that for every \(\varepsilon>0\) there is \(\delta>0\) such that, whenever \(z\in B_X\) satisfies
\[
f_u(z)>1-\delta
\]
for some \(u\in U\), then
\[
\|u-z\|<\varepsilon.
\]
\end{definition}

Examples of Banach spaces $X$, where $S_X$ is a uniformly exposed set include uniformly convex Banach spaces. Another example is $\ell_1(\Gamma)$, where $U=\{\pm e_\gamma\colon \gamma\in \Gamma\}$.

We now give a simple characterization of uniformly strongly exposing set in a separable Banach space, the proof it is straigthforward, and thus omitted.
\begin{lemma}\label{lem:countable-use-selection}
Let \(X\) be a separable Banach space, and let \(U\subset S_X\) satisfy
\[
        B_X=\overline{\operatorname{conv}}(U).
\]
Then there is a sequence \((u_n)\subset U\) such that
\[
        B_X=\overline{\operatorname{conv}}\{u_n:n\in\mathbb N\}.
\]
In particular,
\[
        \sup_{n\in\mathbb N}|x^*(u_n)|=\|x^*\|
        \qquad (x^*\in X^*).
\]
If, moreover, \(U\) is uniformly strongly exposed by a family
\((f_u)_{u\in U}\), then \((u_n)\) is uniformly strongly exposed by
\((f_{u_n})\).
\end{lemma}



\subsection{Ball-covering properties}
We will now give the definitions of the different versions of ball-covering properties and establish some preliminary results on them. We begin by noting that ball-covering properties can be equivalently defined with closed balls as well. We start with the original ball-covering property introduced by L.~Cheng in 2006.

\begin{definition}[\cite{Cheng2006}]
    Let $X$ be a Banach space. Then $X$ is said to have the \emph{ball-covering property} (BCP, in short) if there is a countable collection of balls $\{B(x_n,r_n)\}$ with the following properties:
    \begin{enumerate}
        \item[(1)] $S_X\subset \bigcup_{n\in\mathbb{N}}B(x_n,r_n)$;
        \item[(2)] $\|x_n\|\geq r_n$.
    \end{enumerate}
\end{definition} 

Z.~Luo and B.~Zheng introduced in 2021 the following strengthenings of the classical ball-covering property.

\begin{definition}[\cite{LuoZheng2021}]
    Let $X$ be a Banach space and $r>0$. Then $X$ is said to have the \emph{$r$-strong ball-covering property} ($r$-SBCP, in short) if there is a countable collection of balls $\{B(x_n,r_n)\}$ with the following properties:
    \begin{enumerate}
        \item[(1)] $S_X\subset \bigcup_{n\in\mathbb{N}}B(x_n,r_n)$;
        \item[(2)] $\|x_n\|\geq r_n$;
        \item[(3)] $\sup_{n\in\mathbb{N}}r_n\leq r$.
    \end{enumerate}
    The space $X$ is said to have the \emph{strong ball-covering property} (SBCP, in short) if it has the $r$-SBCP for some $r>0$.
\end{definition} 

\begin{definition}[\cite{LuoZheng2021}]
    Let $X$ be a Banach space and $r,\alpha>0$. Then $X$ is said to have the \emph{$(r,\alpha)$-uniform ball-covering property} ($(r,\alpha)$-UBCP, in short) if there is a countable collection of balls $\{B(x_n,r_n)\}$ with the following properties:
    \begin{enumerate}
        \item[(1)] $S_X\subset \bigcup_{n\in\mathbb{N}}B(x_n,r_n)$;
        \item[(2')] $\|x_n\|\geq r_n+\alpha$;
        \item[(3)] $\sup_{n\in\mathbb{N}}r_n\leq r$.
    \end{enumerate}
    The space $X$ is said to have the \emph{uniform ball-covering property} (UBCP, in short) if it has the $(r,\alpha)$-UBCP for some $r,\alpha>0$.
\end{definition} 

From the definitions it is clear that
\[
\text{UBCP $\implies$ SBCP $\implies$ BCP.}
\]

We now give an equivalent formulation of the UBCP, which will be useful to us in Section~\ref{sec: Lipschitz spaces}, and a variant of it can be found in \cite{LuoZheng2021}.

\begin{lemma}[cf. {\cite[Lemmata~2.1 and 2.2]{LuoZheng2021}}]\label{lem:ubcp-normal-form}
Let \(X\) be a Banach space with the UBCP. Then there exist numbers
\(R>\alpha>0\) and a sequence \((g_n)\subset S_X\) such that
\[
        S_X\subset \bigcup_{n=1}^\infty B(Rg_n,R-\alpha).
\]
More precisely, if \(X\) has the \((r,\alpha)\)-UBCP, then every
\(R>r+1\) can be used.
\end{lemma}

The following extension of the ball-covering property was introduced by A.~Guirao, A.~Lissitsin, and V.~Montesinos in 2019.

\begin{definition}[\cite{GLM2019}]
    Let $X$ be a Banach space and $\alpha\in[-1,1)$. Then $X$ is said to have the \emph{strictly $\alpha$-off ball-covering property} ($\alpha$-BCP, in short) if there is a countable collection of balls $\{B(x_n,r_n)\}$ with the following properties:
    \begin{enumerate}
        \item[(1)] $S_X\subset \bigcup_{n\in\mathbb{N}}B(x_n,r_n)$;
        \item[(2')] $\|x_n\|\geq r_n+\alpha$.
    \end{enumerate}
\end{definition}

Figure~\ref{fig:relations-between-ball-covering-properties} illustrates the relationship between different ball-covering properties. The examples that differentiate ball-covering properties from each other can be found in \cite{GLM2019} and \cite{LuoZheng2021}. 

\begin{figure}[H]
\centering
\begin{tikzpicture}

\def\arrsep{3mm}

\tikzset{
    propbox/.style={
        draw,
        text width=3.25cm,
        align=center,
        outer sep=5pt,
        minimum height=8mm
    },
    sidebox/.style={
        draw,
        text width=3cm,
        align=center,
        outer sep=5pt,
        minimum height=8mm
    },
    impl/.style={
        -{Implies},
        double,
        double distance=2.4pt,
        shorten <=2.5mm,
        shorten >=2.5mm
    },
    nonimpl/.style={
        -{Implies},
        double,
        double distance=2.4pt,
        shorten <=2.5mm,
        shorten >=2.5mm
    }
}

\node[propbox] (B) at (0,0) {BCP\\[-1mm] {\((\alpha=0)\)}};
\node[propbox] (S) at (0,2.5) {$r$-SBCP};
\node[propbox] (U) at (0,5.0) {$(r,\alpha)$-UBCP};

\node[sidebox] (L) at (-5.6,0) {$\alpha$-BCP\\[-1mm]$(\alpha<0)$};
\node[sidebox] (G) at (5.6,0) {$\alpha$-BCP\\[-1mm]$(\alpha>0)$};


\draw[impl]
    ([xshift=-\arrsep]U.south) --
    ([xshift=-\arrsep]S.north);

\draw[impl]
    ([xshift=-\arrsep]S.south) --
    ([xshift=-\arrsep]B.north);

\draw[nonimpl]
    ([xshift=\arrsep]B.north) --
    coordinate (nBS)
    ([xshift=\arrsep]S.south);

\draw[nonimpl]
    ([xshift=\arrsep]S.north) --
    coordinate (nSU)
    ([xshift=\arrsep]U.south);

\draw[line width=0.6pt] ($(nBS)+(-4pt,-4pt)$)--($(nBS)+(4pt,4pt)$);
\draw[line width=0.6pt] ($(nSU)+(-4pt,-4pt)$)--($(nSU)+(4pt,4pt)$);


\draw[impl]
    ([yshift=\arrsep]B.west) --
    ([yshift=\arrsep]L.east);

\draw[impl]
    ([yshift=\arrsep]G.west) --
    ([yshift=\arrsep]B.east);

\draw[nonimpl]
    ([yshift=-\arrsep]L.east) --
    coordinate (nLB)
    ([yshift=-\arrsep]B.west);

\draw[nonimpl]
    ([yshift=-\arrsep]B.east) --
    coordinate (nBG)
    ([yshift=-\arrsep]G.west);

\draw[line width=0.6pt] ($(nLB)+(-4pt,-4pt)$)--($(nLB)+(4pt,4pt)$);
\draw[line width=0.6pt] ($(nBG)+(-4pt,-4pt)$)--($(nBG)+(4pt,4pt)$);


\draw[impl]
    (U.east) -- (G.north);


\draw[nonimpl]
    ([xshift=-6mm,yshift=1mm]G.north) --
    coordinate (nGS)
    ([xshift=1mm,yshift=-1mm]S.east);

\draw[line width=0.6pt] ($(nGS)+(-5pt,-5pt)$)--($(nGS)+(5pt,5pt)$);

\end{tikzpicture}
\caption{Relations between ball-covering properties.}
\label{fig:relations-between-ball-covering-properties}

\end{figure}
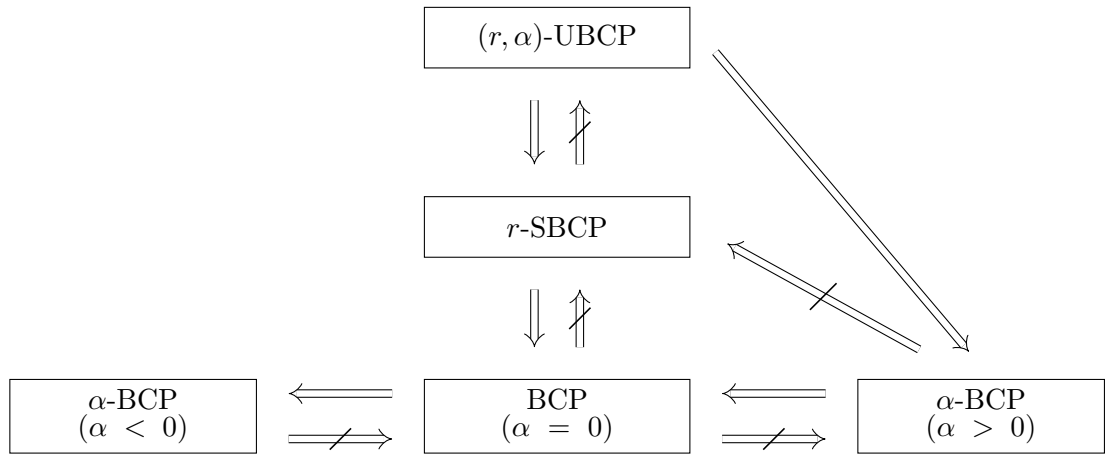

Following \cite{GLM2019}, we denote the \emph{ball-covering index} of a Banach space $X$ by
\[
 \operatorname{bc}(X)
 :=
 \sup\{\alpha\in[-1,1): X \text{ has the } \alpha\text{-BCP}\}.
\]
We use the convention that this supremum is equal to \(-1\) when the set
is empty.

We now collect some useful facts about ball-covering properties, which will be used throughout the rest of the paper.

\begin{lemma}[\cite{GLM2019}]\label{lem:basic-implications}
Let \(X\) be a Banach space.
\begin{enumerate}
\item[(a)] If $X^*$ has the BCP, then $X$ is separable.
\item[(b)] If \(X\) has the \(\alpha\)-BCP and \(-1\leq \beta\leq \alpha<1\),
then \(X\) has the \(\beta\)-BCP.
\item[(c)] If \(X\) has the \((r,\alpha)\)-UBCP, then \(X\) has the \(r\)-SBCP
and the \(\alpha\)-BCP.
\item[(d)] Consequently,
\[
        bc(X)=1
        \quad\Longleftrightarrow\quad
        X \text{ has the } \alpha\text{-BCP for every } -1\leq \alpha<1 .
\]
\end{enumerate}
\end{lemma}

We will now recall the following lemma, which directly follows from the discussion in \cite[Page~611]{GLM2019}.
  
\begin{lemma}[\cite{GLM2019}]\label{lem:homogeneous-alpha-bcp}
Let \(X\) be a Banach space and let \(\alpha\in [0,1)\). If
\(\operatorname{bc}(X)=1\), then there exists a countable set \(A\subset X\)
such that for every \(x\in X\setminus\{0\}\) there is \(a\in A\) satisfying
\[
        \|a\|-\|a-x\|>\alpha\|x\|.
\]
\end{lemma}

It is straightforward to check that $\ell_1(\Gamma)$ is a space without the $(-1)$-BCP whenever $\Gamma$ is uncountable (see \cite[Section~5]{GLM2019}). The proof of \cite[Theorem~2.3]{CKWZ2010}, given there for $\ell_1[0,1]$, uses only the failure of \((-1)\)-BCP and yields the following result. 

\begin{proposition}\label{prop: ball-coverings with big radius}
Let \((X,\|\cdot\|)\) be a Banach space without the \((-1)\)-BCP, let
\(\varepsilon>0\), and let \(|\cdot|\) be an equivalent norm on \(X\)
such that
\[
        \|x\|\leq |x|\leq (1+\varepsilon)\|x\|
        \qquad (x\in X).
\]
Then, for every countable family of open balls $\bigl\{B_{|\cdot|}(x_n,r_n)\colon n\in\mathbb N\}$
satisfying
\[
        S_{(X,|\cdot|)}
        \subset \bigcup_{n=1}^\infty B_{|\cdot|}(x_n,r_n)
        \quad\text{and}\quad
        |x_n|\geq r_n \quad(n\in\mathbb N),
\]
one has
\[
        \sup_{n\in\mathbb N} r_n>\frac1\varepsilon .
\]
\end{proposition}

In \cite[Example 3.1]{LuoZheng2021} the authors used $\ell_1[0,1]$ and \cite[Theorem~2.3]{CKWZ2010} to construct a Banach space with the BCP but without the SBCP. With the help of Proposition~\ref{prop: ball-coverings with big radius}, we can generalize and improve this result by constructing many dual Banach spaces $X^*$ with even $\bc (X^*)=1$ but without the SBCP (see Proposition~\ref{prop: property A does not imply SBCP}).

\section{Main results}\label{sec: main results}
In this section we study ball-covering properties of dual Banach spaces in terms of
the geometry of the predual unit ball. In \cite{GLM2019}, A.~J.~Guirao, A.~Lissitsin, and V.~Montesinos proved that if the unit ball of a separable Banach space is a closed convex hull of its strongly exposed points, then $\bc(X^*)=1$. Our first main result (see Theorem~\ref{thm: BCP in dual}) gives a precise dual
characterization of the optimal ball-covering index: if \(X\) is separable, then
\[
        B_X \text{ is an SCD set}
        \quad\Longleftrightarrow\quad
        \operatorname{bc}(X^*)=1 .
\]
We then show that this phenomenon is genuinely different from the ordinary
ball-covering property. Namely, there are equivalent dual norms for which the
dual space has the BCP but the corresponding value of \(\operatorname{bc}(\cdot)\) is strictly
smaller than one; equivalently, the predual unit ball is not SCD (see Proposition~\ref{prop:dual-bcp-but-bc-less-than-one}). 

After that we turn to uniform versions of the ball-covering property. We prove that if \(X\) has
a \(K\)-unconditional basis with \(K<2\), then \(X^*\) has the UBCP (see Proposition~\ref{prop:k-unconditional-dual-bcp}). The constant
\(2\) is sharp: using a recent construction of M.~L\~oo and Y.~Perreau \cite{LP2026}, we show that their
space is an example of a Banach space with a \(2\)-unconditional basis whose dual fails even the BCP (see Proposition~\ref{prop:two-unconditional-dual-no-bcp}).

Finally, we give a second source of dual spaces with the UBCP, based on
uniformly strongly exposed sets. If \(B_X\) is the closed convex hull of a set of
uniformly strongly exposed points, then \(X^*\) has the UBCP and
\(\operatorname{bc}(X^*)=1\) (see Theorem~\ref{thm: UBCP in dual}). This should be compared with Lindenstrauss property
A: for separable spaces, property A implies \(\operatorname{bc}(X^*)=1\), but we show
that it does not in general imply the SBCP (see Proposition~\ref{prop: property A does not imply SBCP}).

\subsection{BCP in dual spaces}

\begin{theorem}\label{thm: BCP in dual}
    Let $X$ be a separable Banach space. Then $B_X$ is an SCD set if and only if $\bc(X^*)=1$.
\end{theorem}
\begin{proof}
    Assume first that $B_X$ is an SCD set, hence there is a sequence of slices $\{S_n\colon n\in\mathbb N\}$ of $B_X$, where $S_n:=\{x\in B_X\colon f_n(x)>1-\gamma_n\}$ for some $f_n\in S_{X^*}$ and $\gamma_n>0$. 

    Let $\alpha\in[0,1)$. We will show that $X^*$ has the $\alpha$-BCP. Choose $\varepsilon>0$ so that $\alpha+\varepsilon<1$. For each \(n\), we  choose a large enough natural number \(m_n\) so that 
\begin{equation}\label{eq: choosing m_n}
    \frac{1+\alpha+\varepsilon}{m_n}<\gamma_n.
\end{equation}

 We claim that
\begin{equation}\label{eq: covering the dual unit sphere}
S_{X^*} \subset \bigcup_{n=1}^{\infty} B\left(m_nf_n,\,m_n-\alpha\right).
\end{equation}    
Suppose for the sake of contradiction that there exists some functional
\(f \in S_{X^*}\) such that
for every \(n\):
\[
\|m_n f_n - f\| \ge m_n - \alpha.
\]
Thus, for every \(n\), there must exist some point \(b_n \in B_X\)
such that:
\begin{equation}\label{eq: b_n norms m_n f_n - f}
(m_n f_n - f)(b_n)
> m_n - \alpha - \varepsilon.
\end{equation}
Because \(f \in S_{X^*}\) and \(b_n \in B_X\), we know that
$ f(b_n) \ge -1,$ hence \eqref{eq: b_n norms m_n f_n - f} implies that
\[
 f_n(b_n)
> \frac{1}{m_n}\left(m_n - \alpha -\varepsilon +f(b_n)\right)\geq 1-\frac{1+\alpha+\varepsilon}{m_n}>1-\gamma_n.
\]
Therefore, $b_n \in S_n$. On the other hand, since $f_n(b_n)\leq 1$, we 
get from \eqref{eq: b_n norms m_n f_n - f} that
\[
 f(b_n) < \alpha +\varepsilon.
\]
We have constructed a sequence of points
$\{b_n\}_{n=1}^{\infty}\subset B_X$ such that \(b_n\in S_n\) for every \(n\).
Because the sequence of slices \(\{S_n\}\) is determining for the set
\(B_X\), 
\[
B_X=\overline{\conv}(\{b_n\colon n\in\mathbb N\}).
\]
However, this leads to a contradiction, because
\[
1=\|f\|=\sup_{x\in B_X}f(x)=
\sup_{n\in \mathbb N} f(b_n)\leq \alpha+\varepsilon<1.
\]
Thus, \eqref{eq: covering the dual unit sphere} holds, and \(X^*\) has the \(\alpha\)-BCP. Since \(\alpha \in [0,1)\) was arbitrary, $\bc(X^*) = 1$.

Assume now that \(\bc(X^*)=1\). For each \(n\in\mathbb N\), put
\[
        \alpha_n:=1-\frac1n.
\]
Since \(X^*\) has the \(\alpha_n\)-BCP, there exist
\(f_{n,j}\in X^*\) and \(r_{n,j}>0\), \(j\in\mathbb N\), such that
\[
        S_{X^*}\subset
        \bigcup_{j=1}^{\infty} B(f_{n,j},r_{n,j})
\]
and
\[
        r_{n,j}+\alpha_n\leq \|f_{n,j}\|
        \qquad(j\in\mathbb N).
\]
For \(n,j\in\mathbb N\) and \(q\in\mathbb Q\) with
\(q<\|f_{n,j}\|\), define
\[
        S_{n,j,q}:=\{x\in B_X:f_{n,j}(x)>q\}.
\]
We claim that the countable family
\[
        \mathcal S:=
        \{S_{n,j,q}:n,j\in\mathbb N,\ q\in\mathbb Q,\
        q<\|f_{n,j}\|\}
\]
is determining for \(B_X\).
Let
\[
        S:=\{x\in B_X:g(x)>a\}
\]
be an arbitrary slice of \(B_X\), where \(g\in S_{X^*}\) and \(a<1\).
Choose \(n\in\mathbb N\) such that
\(
        a<\alpha_n.
\)
Since the balls \(B(f_{n,j},r_{n,j})\) cover \(S_{X^*}\), there is
\(j\in\mathbb N\) such that
\[
        \|g-f_{n,j}\|<r_{n,j}.
\]
Moreover,
\[
        r_{n,j}+a<r_{n,j}+\alpha_n\leq\|f_{n,j}\|.
\]
Choose \(q\in\mathbb Q\) such that
\[
        r_{n,j}+a<q<\|f_{n,j}\|.
\]
If \(x\in S_{n,j,q}\), then
\[
        g(x)
        =
        f_{n,j}(x)+(g-f_{n,j})(x)
        >
        q-\|g-f_{n,j}\|
        >
        a.
\]
Thus \(S_{n,j,q}\subset S\). Hence every slice of \(B_X\) contains a
member of \(\mathcal S\), so \(B_X\) is an SCD set.

\end{proof}

\begin{remark}
In \cite{LLMRZ2024}, the notion of an SCD-point was introduced. By carefully analyzing the proof of Theorem~\ref{thm: BCP in dual}, a similar argument shows that: if $X^*$ has the BCP, then $0$ is an SCD-point of $B_X$. 
\end{remark}

We will now show that there are separable Banach spaces $X$ such that $X^*$ has the BCP, but $B_X$ is not an SCD set. It is known that every Banach space with the Daugavet property is such that its dual space fails the $(-1)$-BCP \cite[Lemma 5.1]{CLL2023}, hence the next proposition can be applied to, e.g. $C[0,1]$ or $L_1[0,1]$.

\begin{proposition}\label{prop:dual-bcp-but-bc-less-than-one}
Let $X$ be a separable Banach space such that $X^*$ fails the $(-1)$-BCP. Then there exists an equivalent norm $\vertiii{\cdot}$ on $X$ such that $(X^*, \vertiii{\cdot}^*)$ has the BCP, but \[\operatorname{bc}((X^*, \vertiii{\cdot}^*))<1.
\]
In particular, $B_{(X, \vertiii{\cdot})}$ is not an SCD set.
\end{proposition}

\begin{proof}
Let \(\|\cdot\|_0\) denote the original norm on $X^*$. By \cite[Theorem 15]{GLM2019}, there exists an equivalent dual norm
\(\|\cdot\|_g\) on \(X^*\) such that
\[
        \operatorname{bc}(X^*,\|\cdot\|_g)=1.
\]
Choose \(C\geq 1\) such that
\[
        \frac1C\|\cdot\|_g\leq \|\cdot\|_0\leq C\|\cdot\|_g.
\]
Fix \(\alpha\in(0,1)\) and choose
\(
        \lambda\in \left(\frac{C}{C+\alpha},1\right).
\)
For each \(x^*\in X^*\), define
\[
\vertiii{x^*}^*:
=
\lambda\|x^*\|_g
+
(1-\lambda)\|x^*\|_0.
\]
Since \(\|\cdot\|_0\) and \(\|\cdot\|_g\) are dual norms on the same dual
pair \((X^*,X)\), their convex combination \(\vertiii{\cdot}^*\) is also a dual norm. Thus $Y:=(X^*,\vertiii{\cdot}^*)$ is a dual Banach space with a separable predual.

We first show that \(Y\) has the BCP. Since
\(
        \operatorname{bc}(X^*,\|\cdot\|_g)=1,
\)
Lemma~\ref{lem:homogeneous-alpha-bcp}, applied to
\((X^*,\|\cdot\|_g)\), gives a countable set
\(A\subset X^*\) such that for every
\(y\in S_{(X^*,\vertiii{\cdot}^*)}\) there is \(a\in A\) satisfying
\begin{equation}\label{eq: alpha-BCP}
    \|a\|_g-\|a-y\|_g>\alpha\|y\|_g.
\end{equation}
For such \(a\) and \(y\), we also have
\begin{equation}\label{eq: equivalent norms}
        \|a\|_0-\|a-y\|_0
        \geq -\|y\|_0
        \geq -C\|y\|_g.
\end{equation} 
Hence, by \eqref{eq: alpha-BCP} and \eqref{eq: equivalent norms},
\[
\begin{aligned}
\vertiii{a}^*-\vertiii{a-y}^*
        &=
        \lambda\bigl(\|a\|_g-\|a-y\|_g\bigr)
        +(1-\lambda)\bigl(\|a\|_0-\|a-y\|_0\bigr)  \\
        &>
        \bigl(\lambda\alpha-(1-\lambda)C\bigr)\|y\|_g\\
        &>
        0.
\end{aligned}
\]
Therefore
\(
        y\in B(a,\vertiii{a}^*).
\)
Thus
\[
        S_Y\subset \bigcup_{a\in A} B(a,\vertiii{a}^*),
\]
and so \(Y\) has the BCP.

It remains to prove that \(\operatorname{bc}(Y)<1\). Let
\(B\subset X^*\) be an arbitrary countable set. Since $(X^*,\|\cdot\|_0)$ fails the $(-1)$-BCP, by
\cite[Proposition~2.3]{CLL2023} there exists \(u\in X^*\) with
\(\|u\|_0=1\) such that, for every \(b\in B\) and \(s\in\mathbb R\),
\begin{equation}\label{eq: fail -1-BCP}
\|b+su\|_0=\|b\|_0+|s|.
\end{equation}
Let \(z:=u/\vertiii{u}^*\). For every \(b\in B\), by the triangle inequality and by \eqref{eq: fail -1-BCP}, we have
\[
\begin{aligned}
        \vertiii{b}^*-\vertiii{b-z}^*
        &=
        \lambda\bigl(\|b\|_g-\|b-z\|_g\bigr)
        +(1-\lambda)\bigl(\|b\|_0-\|b-z\|_0\bigr)  \\
        &\leq
        \lambda\|z\|_g-(1-\lambda)\|z\|_0  \\
        &=
        \frac{\lambda\|u\|_g-(1-\lambda)}
             {\lambda\|u\|_g+(1-\lambda)}\\
    &\leq \frac{\lambda C-(1-\lambda)}
             {\lambda C+(1-\lambda)}.
\end{aligned}
\]
Therefore 
\[
        \bc(Y)
        \leq
        \frac{\lambda C-(1-\lambda)}
             {\lambda C+(1-\lambda)}
        <1.
\]

Finally, that the unit ball of the separable predual $Y_*$ is not an SCD set follows immediately from Theorem \ref{thm: BCP in dual}, because $\bc (Y)<1$.
\end{proof}


\subsection{UBCP in dual spaces}

V.~Kadets, M.~Mart\'in, J.~Mer\'i, and D.~Werner proved that the unit ball is an SCD set for every Banach space with a 1-unconditional basis \cite[Theorem~3.1]{KMMW2013}. On the other hand, M.~L{\~o}o and Y.~Perreau very recently showed that there exists a Banach space with a 2-unconditional basis whose unit ball is not an SCD set \cite[Theorem~3.11]{LP2026}. Hence, the unconditional basis constant is important to determine whether the unit ball is an SCD set. Moreover, it is unknown whether the unit ball of a Banach space with a \(K\)-unconditional basis, where $K\in(1,2)$, is an SCD set \cite[Question~5.4]{LP2026}. However, we can show that the dual of such a space has the UBCP, which is a consequence of the following result.

\begin{proposition}\label{prop:k-unconditional-dual-bcp}
If \(X\) is a Banach space with a \(K\)-unconditional basis
\((e_n)\), where \(K\in[1,2)\), then \(X^*\) has the UBCP.
\end{proposition}

\begin{proof}
Let \((e_n)\) be a \(K\)-unconditional basis of \(X\), and for \(N\in\mathbb N\) let
\[
P_Nx:=\sum_{n=1}^N e_n^*(x)e_n
\]
be the natural projections associated with the basis.
Since the basis is \(K\)-unconditional, for every \(N\in\mathbb N\) we have
\[
        \|I-2P_N^*\|
        =
        \|I-2P_N\|
        \le K .
\]
Therefore, $\|P_N^*\|\le \frac{1+K}{2}$ for every $N\in\mathbb N$. 
Since $P_N^*X^*$ is finite-dimensional, the set
\(
P_N^*X^*\cap \frac{1+K}{2}B_{X^*}
\)
is separable. Hence, it contains a countable dense subset, which
we denote by $D_N$.
Put
\[
D:=\bigcup_{N=1}^\infty D_N.
\]
Fix $\eta\in (0,2-K)$. We claim that 
\begin{equation}\label{eq: UBCP for K-unc}
S_{X^*}\subset \bigcup_{\substack{v\in D\\ 2\|v\|>K+\eta}} B\left(2v,2\|v\|-\eta\right).
\end{equation}
Indeed, let $f\in S_{X^*}$. Since $\sup_N \|P_N^*f\|\ge 1$, we can choose $N\in \mathbb N$ such that 
\[
\|P_N^*f\|>\frac{K+\eta}{2}.
\]
Now we have
\[
\|f-2P_N^*f\|\le K<2\|P_N^*f\|-\eta.
\]
Observe that
\[
P^*_Nf\in P_N^*X^*\cap \left(\frac{1+K}{2}\right)B_{X^*}.
\]

By the density of $D_N$ in $P_N^*X^*\cap \left(\frac{1+K}{2}\right)B_{X^*}$, we can choose $v\in D_N$ sufficiently close to $P_N^*f$ so that

\[
\|v\|>\frac{K+\eta}{2}\qquad \text{and} \qquad \|f-2v\|<2\|v\|-\eta.
\]

Hence, $f\in B(2v, 2\|v\| - \eta)$ for this $v \in D$ and \eqref{eq: UBCP for K-unc} holds. Finally, since 
\[
2\|v\|-\eta \leq 1+K-\eta,
\]
then \(X^*\) has $(1+K-\eta,\eta)$-UBCP.
\end{proof}


We will now show that Proposition~\ref{prop:k-unconditional-dual-bcp} is sharp in the sense that the constant \(2\) cannot be included in its assumption. The following example is exactly the same space as in \cite[Theorem~3.11]{LP2026}, where the authors showed that its unit ball fails to be SCD. We show that the dual of this space even fails the BCP.

\begin{proposition}\label{prop:two-unconditional-dual-no-bcp}
There exists a Banach space with a \(2\)-unconditional basis whose dual fails the BCP.
\end{proposition}

\begin{proof}
We follow the notation and construction found in \cite[Section~3]{LP2026}. Let
\[
    T:=\{\emptyset\}\cup\bigcup_{n=1}^{\infty}\{0,1\}^n
\]
be the infinite binary tree. We write \(s\preceq t\) if \(s\) is an initial
segment of \(t\). A branch is a maximal chain, and an antichain is a set of
pairwise incomparable nodes. Let \(X:=X_T\) be the completion of \(c_{00}(T)\)
with respect to the norm
\[
\|x\|_X
=
\sup\left\{
\sum_{t\in C}|x(t)|:
C\subset T \text{ is a chain}
\right\}.
\]
Let \(B_X^+:=B_X\cap X^+\), and let \(Y:=(X,|||\cdot|||)\), where
\(
    B_Y:=\cconv(B_X^+\cup(-B_X^+)).
\)
By \cite[Theorem~3.11]{LP2026}, \(Y\) has a \(2\)-unconditional basis.

We show that \(Y^*\) fails the BCP. By \cite[Lemma~3.8]{LP2026}, we have
\[
    B_X^+=\cconv(\Sigma^+),
\]
where
\[
\Sigma^+:
=
\left\{
\sum_{t\in A}e_t:
A\subset T \text{ is a finite antichain}
\right\}.
\]
Hence, for every \(h\in Y^*\),
\begin{equation}\label{eq:dual-tree-norm}
\|h\|_{Y^*}=\sup_{x\in B_X^+}|h(x)|=\sup_{u\in \Sigma^+}|h(u)|=\sup\left\{
\left|\sum_{t\in A}h(e_t)\right|:A\subset T \text{ is a finite antichain}\right\}.
\end{equation}
We first note that, for every \(h\in Y^*\), the set of branches \(\beta\) such that
\[
    \limsup_{k\to\infty}|h(e_{\beta(k)})|>0
\]
is at most countable. Indeed, fix \(\delta>0\). If
\(\beta_1,\ldots,\beta_m\) are distinct branches with
\[
    \limsup_{k\to\infty}|h(e_{\beta_i(k)})|>\delta
    \qquad(i=1,\ldots,m),
\]
then we may choose pairwise incomparable nodes \(t_i\in\beta_i\) such that
\[
    |h(e_{t_i})|>\delta
    \qquad(i=1,\ldots,m).
\]
Splitting the antichain \(\{t_1,\ldots,t_m\}\) according to the signs of
\(h(e_{t_i})\), and using \eqref{eq:dual-tree-norm}, we get
\[
    m\delta
    <
    \sum_{i=1}^m |h(e_{t_i})|
    \le
    2\|h\|_{Y^*}.
\]
Thus, for fixed \(\delta>0\), there are only finitely many such branches.
Taking \(\delta\in\mathbb Q\), \(\delta>0\), proves the claim.

Now let \((g_n)\subset Y^*\) be an arbitrary countable family. By the claim,
we may choose two distinct branches \(\beta\) and \(\gamma\) such that, for every
\(n\in\mathbb N\),
\[
    g_n(e_t)\to0
    \quad\text{as }t\in\beta,\ |t|\to\infty,
    \qquad
    g_n(e_t)\to0
    \quad\text{as }t\in\gamma,\ |t|\to\infty.
\]
Let
\[
    f:=b_\beta-b_\gamma,
\]
where \(b_\beta(e_t)=1\) if \(t\in\beta\) and \(b_\beta(e_t)=0\) otherwise.
Note that \(f\in S_{Y^*}\). Indeed, if \(A\subset T\) is a finite
antichain, then each branch meets \(A\) in at most one point. Hence
\[
F(A):=\sum_{t\in A}f(e_t)\in\{-1,0,1\}.
\]
Indeed, a point belonging to the common initial part of \(\beta\) and
\(\gamma\) contributes \(0\), a point in \(\beta\setminus\gamma\)
contributes \(1\), and a point in \(\gamma\setminus\beta\) contributes
\(-1\). Since \(A\) is an antichain, no other values are possible.
Therefore, by \eqref{eq:dual-tree-norm}, \(\|f\|_{Y^*}\leq 1\). Since \(\beta\neq\gamma\),
there exists \(t\in\beta\setminus\gamma\), and taking \(A=\{t\}\) gives
\(|F(A)|=1\). Thus \(\|f\|_{Y^*}=1\).

Fix \(n\in\mathbb N\), and write \(g=g_n\). For a finite antichain
\(A\subset T\), put
\[
    G(A):=\sum_{t\in A}g(e_t).
\]

We prove that \(|G(A)|\le \|g-f\|_{Y^*}\) for every finite antichain \(A\).
If \(F(A)=0\), then
\[
    |G(A)|
    =
    \left|
    \sum_{t\in A}(g-f)(e_t)
    \right|
    \le
    \|g-f\|_{Y^*}.
\]
Assume that \(F(A)=1\). Then \(A\) contains one point on
\(\beta\setminus\gamma\) and no point on \(\gamma\setminus\beta\). Moreover, \(A\) cannot contain a
common initial node of \(\beta\) and \(\gamma\), because it already
contains a point of \(\beta\setminus\gamma\). Hence, for every
\(a\in A\), all sufficiently far nodes of \(\gamma\) are incomparable
with \(a\). Thus we may choose \(v\in\gamma\setminus\beta\) arbitrarily
far along \(\gamma\) so that \(A\cup\{v\}\) is still an antichain and \(f(e_v)=-1\). Then
\[
    \sum_{t\in A\cup\{v\}}(g-f)(e_t)
    =
    G(A)+g(e_v).
\]
Hence
\[
    |G(A)+g(e_v)|
    \le
    \|g-f\|_{Y^*}.
\]
Letting \(|v|\to\infty\) along \(\gamma\), we obtain
\[
    |G(A)|\le \|g-f\|_{Y^*}.
\]
The case \(F(A)=-1\) is symmetric. Therefore, by \eqref{eq:dual-tree-norm},
\[
    \|g_n\|_{Y^*}\le \|g_n-f\|_{Y^*}
    \qquad(n\in\mathbb N).
\]

Consequently, if \(\{B(g_n,r_n)\}_{n=1}^\infty\) is any countable
family of open balls with \(r_n\le \|g_n\|_{Y^*}\), then the above
\(f\in S_{Y^*}\) does not belong to any of these balls. Hence \(S_{Y^*}\)
cannot be covered by countably many balls missing the origin. Thus \(Y^*\)
fails the BCP.
\end{proof}
A second, independent source of UBCP comes from norming families of uniformly
strongly exposed points. We will show that if the unit ball of a separable Banach space is the closed convex hull of a set of uniformly strongly exposed points, then the dual space has the UBCP.

Recall that a Banach space \(X\) is said to have the \emph{Lindenstrauss property A} if for every Banach space \(Y\), the set of norm-attaining operators from \(X\) to \(Y\) is dense in the space of all bounded linear operators from \(X\) to \(Y\). Lindenstrauss proved in \cite[Proposition~1]{L1963} that if \(B_X\) is the closed convex hull of a set of uniformly strongly exposed points, then \(X\) has Lindenstrauss property A. The following result shows that the same condition implies that the dual space \(X^*\) has the UBCP.

\begin{theorem}\label{thm: UBCP in dual}
Let \(X\) be a separable Banach space. If \(B_X\) is the closed convex hull of a set of uniformly strongly exposed points, then \(X^*\) has the UBCP and $\bc(X^*)=1$.
\end{theorem}

Actually, Theorem~\ref{thm: UBCP in dual} is a particular case of the more general Theorem~\ref{thm:nu-norming-UBCP-real}, because the assumption that \(B_X\) is the closed convex hull of a set of uniformly strongly exposed points satisfies, by Lemma~\ref{lem:countable-use-selection}, the conditions (1)--(3) in Theorem~\ref{thm:nu-norming-UBCP-real} with $\nu=1$.

\begin{theorem}\label{thm:nu-norming-UBCP-real}
Let \(X\) be a separable Banach space and let \(0<\nu\leq 1\). Assume that
there exist sequences \((u_n)\subset S_X\) and \((u_n^*)\subset S_{X^*}\)
such that
\begin{enumerate}
    \item $u_n^*(u_n)=1$ for all $n\in\mathbb N,$
    \item the sequence \((u_n)\) is uniformly strongly exposed by \((u_n^*)\),
    \item $\sup_{n\in\mathbb N}|x^*(u_n)|\geq \nu\|x^*\|$ for all $x^*\in X^*$.
\end{enumerate}
Then $X^*$ has the UBCP and $\bc(X^*)\geq\nu$.
\end{theorem}

\begin{proof}
Let \(0<\alpha<\nu\). We shall show that \(X^*\) has the
\((\lambda-\alpha,\alpha)\)-UBCP for a suitable \(\lambda>0\). 

Choose \(\eta,\gamma\) such that \(\alpha<\eta<\gamma<\nu\), and set
\[
        \varepsilon:=\frac{\gamma-\eta}{2}.
\]
Since \((u_n)\) is uniformly strongly exposed by \((u_n^*)\), there exists
\(0<\delta<1\) such that
\begin{equation}\label{eq:uniform-exposure-nu}
        x\in B_X,\ n\in\mathbb N,\ u_n^*(x)>1-\delta
        \quad\Longrightarrow\quad
        \|x-u_n\|<\varepsilon .
\end{equation}
Set \(\lambda:=2/\delta\). We claim that
\[
        S_{X^*}\subset
        \bigcup_{\substack{n\in\mathbb N\\ \theta\in\{-1,1\}}}
        B\bigl(\lambda\theta u_n^*,\lambda-\alpha\bigr).
\]

Let \(x^*\in S_{X^*}\). By assumption (3), there exists \(N\in\mathbb N\)
such that \(|x^*(u_N)|>\gamma\). Choose \(\theta\in\{-1,1\}\) such that
\(\theta x^*(u_N)>\gamma\). We will prove that
\[
        \|\lambda\theta u_N^*-x^*\|\leq \lambda-\eta.
\]
Since \(\eta>\alpha\), this implies
 \(x^*\in B(\lambda\theta u_N^*,\lambda-\alpha)\).

Since \(B_X\) is symmetric, it is enough to prove that, for every
\(z\in B_X\),
\[
        \lambda\theta u_N^*(z)-x^*(z)\leq \lambda-\eta.
\]
Assume, towards a contradiction, that there exists \(z\in B_X\) such that
\begin{equation}\label{eq:contradiction-nu}
        \lambda\theta u_N^*(z)-x^*(z)> \lambda-\eta .
\end{equation}
Since \(x^*\in S_{X^*}\) and \(z\in B_X\), we have \(x^*(z)\geq -1\).
Hence
\[
        u_N^*(\theta z)=\theta u_N^*(z)>
        1+\frac{x^*(z)-\eta}{\lambda}
        \geq
        1-\frac{1+\eta}{\lambda}
        =
        1-\frac{1+\eta}{2}\delta
        >
        1-\delta.
\]
Therefore \eqref{eq:uniform-exposure-nu} gives
\[
        |x^*(z)-\theta x^*(u_N)|
        \leq \|z-\theta u_N\|=\|\theta z-u_N\|<\varepsilon,
\]
and consequently
\[
        x^*(z)>
        \theta x^*(u_N)-\varepsilon
        >
        \gamma-\varepsilon
        =
        \frac{\gamma+\eta}{2}
        >
        \eta.
\]
It follows that
\[
        \lambda\theta u_N^*(z)-x^*(z)
        \leq
        \lambda-x^*(z)
        <
        \lambda-\eta,
\]
which contradicts \eqref{eq:contradiction-nu}. Moreover, the same covering also shows that \(X^*\) has the \(\alpha\)-BCP.
Since \(0<\alpha<\nu\) was arbitrary, \(\bc(X^*)\geq\nu\).
\end{proof}

\begin{remark}
The converse of Theorem~\ref{thm:nu-norming-UBCP-real} does not hold. For example, \(\ell_1\) has the UBCP, but $c_0$ has no denting points (hence, also no strongly exposed points).  
\end{remark}

J. Lindenstrauss proved that if a separable Banach space \(X\) has the Lindenstrauss property A, then \(B_X\) is the closed convex hull of its strongly exposed points \cite[Theorem~2]{L1963}. The latter implies that $B_X$ is an SCD set \cite[Proposition~2.8]{AKMMS2010}, and hence Theorem~\ref{thm: BCP in dual} gives that \(\bc(X^*)=1\), in particular $X^*$ has the BCP. On the other hand, Theorem~\ref{thm: UBCP in dual} suggests that perhaps Lindenstrauss property A alone already implies that the dual space has the UBCP. We will now show in Proposition~\ref{prop: property A does not imply SBCP} that in general this is not true.

\begin{definition}[{\cite[Definition~1.2]{S1983}}]
    A Banach space $X$ is said to have \emph{property $\alpha$} if there is a $0\leq\rho<1$ and a family $(x_i,x^*_i)_{i\in I}\subset X\times X^*$ such that
    \begin{enumerate}
        \item $\|x_i\|=\|x^*_i\|=x^*_i(x_i)=1$ for every $i\in I$;
        \item if $i\neq j$, then $|x^*_i(x_j)|\leq \rho$;
        \item $B_X$ is the absolute closed convex hull of $\{x_i\}_{i\in I}$.
    \end{enumerate}
\end{definition}

W. Schachermayer observed that property $\alpha$ implies that the unit ball is the closed convex hull of a uniformly strongly exposed set, hence these spaces also have Lindenstrauss property $A$. It is known that every separable Banach space admits a $(3+\varepsilon)$-equivalent norm with property $\alpha$ \cite[Theorem~4.4]{S1983}. The following lemma follows closely the ideas of W. Schachermayer when he $(1+\varepsilon)$-equivalently renormed $c_0$ to have property $\alpha$ \cite[Proposition~3.1]{S1983}. Since we were not able to find this result for $C[0,1]$ in the literature, we include it here with a proof.

\begin{lemma}\label{lemma: C[0,1] renorming alpha}
    For every $K>1$ there is an equivalent norm $\vertiii{\cdot}$ on $C[0,1]$ such that 
    \[
    \frac{1}{K}\|\cdot\|_\infty\leq \vertiii{\cdot}\leq \|\cdot\|_\infty
    \]
    and $(C[0,1],\vertiii{\cdot})$ has property $\alpha$.
\end{lemma}
\begin{proof}
Since \(C[0,1]\) is separable, we can choose a countable dense subset $(d_j)_{j=1}^\infty$ of \(B_{C[0,1]}\). Let \((g_j)_{j=1}^\infty\) be a sequence in
\(B_{C[0,1]}\) obtained from \((d_j)\) by repeating every \(d_j\)
infinitely many times. Thus, for every \(f\in B_{C[0,1]}\) and every
\(\eta>0\), there are infinitely many \(n\) such that
\[
   \|g_n-f\|_\infty<\eta .
\]

Choose pairwise disjoint open intervals \(I_n\subset (0,1)\), points
\(t_n\in I_n\), and, by Urysohn's lemma, continuous functions \(\varphi_n\colon[0,1]\to[0,1]\)
such that
\[
   \operatorname{supp}\varphi_n\subset I_n
   \qquad\text{and}\qquad
   \varphi_n(t_n)=1 .
\]
Define
\[
   x_n:=(1-\varphi_n)g_n+K\varphi_n
   \qquad(n\in\mathbb N).
\]
Then \(x_n\in C[0,1]\) and $\|x_n\|_\infty\le K.$
Moreover, $x_n(t_n)=K.$ If \(m\neq n\), then \(t_m\notin \operatorname{supp}\varphi_n\), and hence
\[
   x_n(t_m)=g_n(t_m).
\]
In particular,
\[
   |x_n(t_m)|\le 1
   \qquad(m\neq n).
\]

Let
\[
   B:=\aconv\{x_n:n\in\mathbb N\}.
\]
Since
\(\|x_n\|_\infty\le K\) for every \(n\), we have
\[
   B\subset K B_{(C[0,1], \|\cdot\|_\infty)}.
\]

We claim that
\[
   B_{(C[0,1], \|\cdot\|_\infty)}\subset B.
\]
Let \(f\in B_{C[0,1]}\) and let \(\eta>0\). Choose \(d_j\) such that
\[
   \|d_j-f\|_\infty<\frac{\eta}{2}.
\]
Choose \(m\in\mathbb N\) so large that
\[
   \frac{K+1}{m}<\frac{\eta}{2}.
\]
Since \(d_j\) occurs infinitely many times in the sequence \((g_j)\),
we can choose distinct indices \(n_1,\dots,n_m\) such that
\[
   g_{n_1}=\cdots=g_{n_m}=d_j.
\]
Put
\[
   u:=\frac1m\sum_{\ell=1}^m x_{n_\ell}.
\]
Then \(u\in \operatorname{conv}\{x_{n_1},\dots,x_{n_m}\}\subset B\).
For every \(s\in[0,1]\), we have
\[
\begin{aligned}
   u(s)-d_j(s)
   &=\frac1m\sum_{\ell=1}^m
      \bigl(x_{n_\ell}(s)-d_j(s)\bigr)  =\frac1m\sum_{\ell=1}^m
      \varphi_{n_\ell}(s)\bigl(K-d_j(s)\bigr).
\end{aligned}
\]
Because the supports of the functions \(\varphi_{n_\ell}\) are pairwise
disjoint, at most one summand in the last sum is nonzero. Therefore
\[
   |u(s)-d_j(s)|\le \frac{K+1}{m}
   \qquad(s\in[0,1]).
\]
Thus
\[
   \|u-d_j\|_\infty\le \frac{K+1}{m}<\frac{\eta}{2}.
\]
Consequently,
\[
   \|u-f\|_\infty
   \le \|u-d_j\|_\infty+\|d_j-f\|_\infty
   <\eta.
\]
Since \(u\in B\) and \(B\) is closed, it follows that \(f\in B\). Hence
\(B_{C[0,1]}\subset B\), as claimed.

Now let $\vertiii{\cdot}$ be the Minkowski functional of \(B\), then $\vertiii{\cdot}$ is an equivalent norm on \(C[0,1]\).
The inclusions
\[
   B_{(C[0,1], \|\cdot\|_\infty)}\subset B\subset K B_{(C[0,1], \|\cdot\|_\infty)}
\]
give
\[
   \frac{1}{K}\|f\|_\infty\le \vertiii{f}\le \|f\|_\infty
   \qquad(f\in C[0,1]).
\]

Finally, we show that \((C[0,1],\vertiii{\cdot})\) has Schachermayer's property
\(\alpha\). For each \(n\in\mathbb N\), define
\[
   x_n^*:=K^{-1}\delta_{t_n}\in C[0,1]^*.
\]
Then
\[
   x_n^*(x_n)=K^{-1}x_n(t_n)=1.
\]
If \(m\neq n\), then
\[
   |x_n^*(x_m)|
   =K^{-1}|x_m(t_n)|
   =K^{-1}|g_m(t_n)|
   \le K^{-1}.
\]
Since
\[
   B=\aconv\{x_n:n\in\mathbb N\},
\]
we get
\[
\begin{aligned}
   \|x_n^*\|_{\vertiii{\cdot}^*}
   &=\sup_{z\in B}|x_n^*(z)|=\sup_m |x_n^*(x_m)|=1.
\end{aligned}
\]
In particular, since \(x_n\in B\), we have \(\vertiii{x_n}\le1\), while
\[
   \vertiii{x_n}\ge x_n^*(x_n)=1.
\]
Thus
\[
   \vertiii{x_n}=1,
   \qquad
   \|x_n^*\|_{\vertiii{\cdot}^*}=1,
   \qquad
   x_n^*(x_n)=1,
\]
and
\[
   |x_n^*(x_m)|\le K^{-1}<1
   \qquad(m\neq n).
\]
Therefore the family \((x_n,x_n^*)_{n=1}^\infty\) witnesses
Schachermayer's property \(\alpha\) for \((C[0,1],\vertiii{\cdot})\), with constant
$\rho=K^{-1}<1.$
\end{proof}

\begin{remark}
    Note that the same proof as in Lemma~\ref{lemma: C[0,1] renorming alpha} also works for $C(K)$ spaces whenever $K$ is an infinite compact metrizable space.
\end{remark}

\begin{proposition}\label{prop: property A does not imply SBCP}
    There exists a separable Banach space $X$ with Lindenstrauss property A such that $\bc(X^*)=1$, but $X^*$ fails the SBCP.
\end{proposition}
\begin{proof}
    Let $Y:=C[0,1]$. Then $Y^*$ does not have $(-1)$-BCP, because $Y$ has the Daugavet property \cite[Lemma~5.1]{CLL2023}. For every $n\in \mathbb{N}$ let
    $Y_n:=(Y,\vertiii{\cdot}_n)$ be the equivalent renorming of $Y$ from Lemma~\ref{lemma: C[0,1] renorming alpha}, where $K_n:=1+1/n$.

    Denote by 
\[
X:=\left(\bigoplus_{n=1}^\infty Y_{n}\right)_{\ell_1}.
\]
We claim that $X$ is a separable Banach space with Lindenstrauss property A, hence also $\bc(X^*)=1$, but $X^*$ fails the SBCP. 

First, by Lemma~\ref{lemma: C[0,1] renorming alpha}, every $Y_n$ has property $\alpha$, hence also Lindenstrauss property A. Finally, note that Lindenstrauss property A is stable under $\ell_1$-sums \cite[Lemma~2]{PS2000}. So $X$ has  Lindenstrauss property A.

Secondly, $\bc(X^*)=1$, because Lindenstrauss property A implies that $B_X$ is the closed convex hull of its strongly exposed points \cite[Theorem~2]{L1963}, and we can apply Theorem~\ref{thm: BCP in dual}.

Thirdly, \(X^*\) fails the SBCP. Denote by \(\|\cdot\|_0\) the usual
dual norm on \(Y^*=C[0,1]^*\), and denote by \(\|\cdot\|_{n,*}\) the dual
norm on
\[
        Y_n^*=(Y,\vertiii{\cdot}_n)^* .
\]
By Lemma~\ref{lemma: C[0,1] renorming alpha}, for every \(y\in Y\),
\[
        \frac{1}{K_n}\|y\|_\infty
        \leq
        \vertiii{y}_n
        \leq
        \|y\|_\infty.
\]
Passing to dual norms, we get, for every \(y^*\in Y^*\),
\[
        \|y^*\|_0
        \leq
        \|y^*\|_{n,*}
        \leq
        K_n\|y^*\|_0
        =
        \left(1+\frac1n\right)\|y^*\|_0 .
\]
In particular, the norm \(\|\cdot\|_{n,*}\) is
\(\left(1+\frac1n\right)\)-equivalent to \(\|\cdot\|_0\) in the sense
needed in Proposition~\ref{prop: ball-coverings with big radius}. Since \(Y^*\), endowed with the usual dual
norm \(\|\cdot\|_0\), fails the \((-1)\)-BCP, Proposition~\ref{prop: ball-coverings with big radius} applied
with \(\varepsilon=1/n\) gives the following: every ball covering
\[
        \{B_{\|\cdot\|_{n,*}}(a_k,r_k): k\in\mathbb N\}
\]
of \(S_{Y_n^*}\) satisfies
\begin{equation}\label{eq: radius unbounded}
    \sup_{k\in\mathbb N} r_k > n .  
\end{equation}
              
Suppose now, towards a contradiction, that \(X^*\) has the SBCP. Then
\(X^*\) has the \(r\)-SBCP for some \(r>0\). Since
\[
        X^*
        =
        \left(
        \bigoplus_{n=1}^\infty Y_n^*
        \right)_{\ell_\infty},
\]
it follows from \cite[Theorem~2.3]{LuoZheng2021} that each \(Y_n^*\) has the
\(r\)-SBCP. Hence, for every \(n\in\mathbb N\), the sphere
\(S_{Y_n^*}\) admits a countable ball covering whose radii are bounded
above by the same \(r\). But this is in contradiction with \eqref{eq: radius unbounded}, because the lower bounds of radii increase by each factor.
\end{proof}

\begin{remark}
    Note that the space $X$ constructed in Proposition~\ref{prop: property A does not imply SBCP} is another example of a Banach space with Lindenstrauss property A, but without property $\alpha$. Indeed, if it had property $\alpha$, then Theorem~\ref{thm: UBCP in dual} would imply that $X^*$ has the UBCP, but, by Proposition~\ref{prop: property A does not imply SBCP}, it would fail the SBCP, a contradiction.
\end{remark}

We do not know of a geometric condition on $B_X$ which would give the SBCP of $X^*$ and would differ from the assumptions in Proposition~\ref{prop:k-unconditional-dual-bcp} and Theorem~\ref{thm:nu-norming-UBCP-real}.

\begin{question}
    What conditions on $B_X$ are sufficient for $X^*$ to have the SBCP? 
\end{question}

\section{Applications to spaces of operators}\label{sec: L(X,Y)}
The results of the previous section are especially useful for spaces of
operators, because these spaces are naturally duals of projective tensor
products. Recall that, for Banach spaces \(X\) and \(Y\), there are
canonical isometric identifications
\[
        \mathcal L(X,Y^*)
        \cong
        (X\widehat\otimes_\pi Y)^*
        \cong
        \mathcal L(Y,X^*).
\]
Moreover,
\[
        B_{X\widehat\otimes_\pi Y}
        =
        \overline{\operatorname{conv}}(B_X\otimes B_Y).
\]
Thus, for separable \(X\) and \(Y\), the criteria from Section~\ref{sec: main results} reduce
the BCP and UBCP of the operator spaces
\(\mathcal L(X,Y^*)\) and \(\mathcal L(Y,X^*)\) to geometric properties
of the unit ball of \(X\widehat\otimes_\pi Y\). 

\subsection{BCP in spaces of operators}
For the BCP we use stability of SCD sets under projective tensor
products. In particular, if one factor has an SCD unit ball and the
other unit ball is the closed convex hull of its denting points, then
the unit ball of the projective tensor product is again SCD. 
\begin{theorem}\label{thm: BCP in L(X,Y)}
    Let $X$ and $Y$ be separable Banach spaces. If $B_X$ is an SCD set and $B_Y$ is the closed convex hull of its denting points, then $\bc(\mathcal{L}(X,Y^*))=\bc(\mathcal{L}(Y,X^*))=1$. In particular, $\mathcal{L}(X,Y^*)$ and $\mathcal{L}(Y,X^*)$ both have the BCP.
\end{theorem}
\begin{proof}
    Since $B_X$ is an SCD set and $B_Y$ is the closed convex hull of its denting points, we have that $B_{X\pten Y}$ is an SCD set \cite[Corollary~4.3]{LLMRZ2024}. Hence, from the identifications $\mathcal{L}(X,Y^*)\cong(X\pten Y)^*\cong \mathcal{L}(Y,X^*)$, and Theorem~\ref{thm: BCP in dual}, one deduces that $\bc(\mathcal{L}(X,Y^*))=\bc(\mathcal{L}(Y,X^*))=1$.
\end{proof}

The previous theorem can be applied, e.g., to pairs of Banach spaces, where one has a 1-unconditional basis and the other one has the Radon--Nikod\'ym property.

\begin{corollary}
    Let $X$ and $Y$ be separable Banach spaces such that $X$ has a 1-unconditional basis and $Y$ has the Radon--Nikod\'ym property. Then $\mathcal{L}(X,Y^*)$ and $\mathcal{L}(Y,X^*)$ both have the BCP.
\end{corollary}
\begin{proof}
    By \cite[Theorem 3.1]{KMMW2013}, \(B_X\) is an SCD set. Since \(Y\)
has the Radon--Nikodým property, \(B_Y\) is the closed convex hull of
its denting points \cite{P1974}. Thus Theorem~\ref{thm: BCP in L(X,Y)} applies.
\end{proof}

\subsection{UBCP in spaces of operators}
For the UBCP we need a uniform version of the classical description of
strongly exposed points in projective tensor products. Recall that W.~M.~Ruess and C.~P.~Stegall proved in \cite{RS1986} that for Banach spaces $X$ and $Y$, one has that
    \[
    \SE(B_ {X\pten Y})=\SE(B_X)\otimes \SE(B_Y),
    \]
where $\SE(B_X)$ denotes the set of strongly exposed points of $B_X$. Since the result for uniformly strongly exposed subsets seems to be not found in the literature and it is of independent interest due to its connections with norm-attaining theory, we prove directly that uniformly
strongly exposed subsets are stable under elementary tensors. 

\begin{proposition}
Let \(X\) and \(Y\) be Banach spaces, and let
\(U\subset S_X\) and \(V\subset S_Y\) be non-empty sets. Then \(U\) and
\(V\) are uniformly strongly exposed subsets of \(B_X\) and \(B_Y\),
respectively, if and only if
\[
        U\otimes V:=\{u\otimes v:u\in U,\ v\in V\}
\]
is a uniformly strongly exposed subset of
\(B_{X\widehat\otimes_\pi Y}\).
\end{proposition}

\begin{proof}
First assume that \(U\) and \(V\) are uniformly strongly exposed subsets
of \(B_X\) and \(B_Y\), respectively. Thus there are families
\[
        \{f_u:u\in U\}\subset S_{X^*},
        \qquad
        \{g_v:v\in V\}\subset S_{Y^*}
\]
such that
\[
        f_u(u)=1,\qquad g_v(v)=1
\]
for all \(u\in U\), \(v\in V\), and such that both families expose
uniformly. We claim that \(U\otimes V\) is uniformly strongly exposed by
the family
\[
        \{f_u\otimes g_v:u\in U,\ v\in V\}
        \subset S_{(X\widehat\otimes_\pi Y)^*}.
\]
Clearly
\[
        (f_u\otimes g_v)(u\otimes v)=1
\]
and \(\|f_u\otimes g_v\|=1\).

Let \(\varepsilon>0\). Choose \(0<\eta<\varepsilon/8\). By the uniform
strong exposure of \(U\) and \(V\), there is \(0<\theta<1\) such that
\[
        x\in B_X,\ f_u(x)>1-\theta
        \quad\Longrightarrow\quad
        \|x-u\|<\eta
\]
for every \(u\in U\), and
\[
        y\in B_Y,\ g_v(y)>1-\theta
        \quad\Longrightarrow\quad
        \|y-v\|<\eta
\]
for every \(v\in V\).

Choose \(0<\delta<1\) so small that
\begin{equation}\label{eq: delta small}
      2\eta(1+\delta)+\frac{4\delta}{\theta}+\delta<\varepsilon.
\end{equation}
      
We show that this \(\delta\) works for the family
\(\{f_u\otimes g_v \colon u\in U, v\in V\}\).

Fix \(u\in U\), \(v\in V\), and let
\(w\in B_{X\widehat\otimes_\pi Y}\) satisfy
\[
        (f_u\otimes g_v)(w)>1-\delta.
\]
Choose a representation
\[
        w=\sum_{k=1}^\infty c_k x_k\otimes y_k,
\]
where \(c_k\ge 0\), \(x_k\in S_X\), \(y_k\in S_Y\), and
\begin{equation}\label{eq: M leq 1+delta}
   M:=\sum_{k=1}^\infty c_k\le 1+\delta. 
\end{equation}

Denote
\[
        \alpha_k:=f_u(x_k),\qquad \beta_k:=g_v(y_k).
\]
Then \(|\alpha_k|\le 1\), \(|\beta_k|\le 1\), and
\begin{equation}\label{eq: M>1-delta}
          \sum_{k=1}^\infty c_k\alpha_k\beta_k
        =
        (f_u\otimes g_v)(w)
        >
        1-\delta.  
\end{equation}

Consequently,
\[
        \sum_{k=1}^\infty c_k(1-\alpha_k\beta_k)
        =
        M-\sum_{k=1}^\infty c_k\alpha_k\beta_k
        <
        (1+\delta)-(1-\delta)
        =
        2\delta.
\]

Let
\[
        G=\{k:\alpha_k\beta_k>1-\theta\}.
\]
With the set $G$ we can distinguish two cases:
\begin{enumerate}
    \item[(a)] Let \(k\in G\). Since \(\alpha_k,\beta_k\in[-1,1]\) and
\(\alpha_k\beta_k>1-\theta>0\), either
\begin{enumerate}
    \item[(1)] $\alpha_k>1-\theta
        \quad\text{and}\quad
        \beta_k>1-\theta$, or
    \item[(2)] $\alpha_k<-1+\theta
        \quad\text{and}\quad
        \beta_k<-1+\theta.$ 
\end{enumerate}
In the first case, the uniform strong exposure of \(U\) and \(V\) gives
\[
        \|x_k-u\|<\eta,
        \qquad
        \|y_k-v\|<\eta.
\]
In the second case,
\[
        f_u(-x_k)>1-\theta,
        \qquad
        g_v(-y_k)>1-\theta,
\]
and hence
\[
        \|-x_k-u\|<\eta,
        \qquad
        \|-y_k-v\|<\eta.
\]
But
\[
        x_k\otimes y_k=(-x_k)\otimes(-y_k),
\]
so in both cases we have
\begin{equation}\label{eq: u times v close to x_k times y_k}
            \|x_k\otimes y_k-u\otimes v\|_\pi<2\eta.
\end{equation}

    \item[(b)] If \(k\notin G\), then \(1-\alpha_k\beta_k\ge \theta\). Hence
\begin{equation}\label{eq: bad set}
           \sum_{k\notin G}c_k
        \le
        \frac{1}{\theta}
        \sum_{k=1}^\infty c_k(1-\alpha_k\beta_k)
        <
        \frac{2\delta}{\theta}. 
\end{equation}
\end{enumerate}

Observe that \eqref{eq: M leq 1+delta} and \eqref{eq: M>1-delta} together imply that $|M-1|\leq\delta$. Therefore, by \eqref{eq: u times v close to x_k times y_k}, \eqref{eq: bad set}, and \eqref{eq: delta small}, we can conclude that
\[
\begin{aligned}
\|w-u\otimes v\|_\pi
&\le
\|w-Mu\otimes v\|_\pi+|M-1|  \\
&\le
\sum_{k=1}^\infty c_k
   \|x_k\otimes y_k-u\otimes v\|_\pi
   +|M-1|                                      \\
&\le
\sum_{k\in G}c_k
   \|x_k\otimes y_k-u\otimes v\|_\pi
+
\sum_{k\notin G}c_k
   \|x_k\otimes y_k-u\otimes v\|_\pi
+
|M-1|                                          \\
&<
2\eta M
+
2\sum_{k\notin G}c_k
+
|M-1|                                          \\
&<
2\eta(1+\delta)
+
\frac{4\delta}{\theta}
+
\delta\\&
<\varepsilon.
\end{aligned}
\]
Thus \(U\otimes V\) is a uniformly strongly exposed
subset of \(B_{X\widehat\otimes_\pi Y}\).

Conversely, assume that \(U\otimes V\) is a uniformly strongly exposed
subset of \(B_{X\widehat\otimes_\pi Y}\). Hence, for each
\(u\in U\) and \(v\in V\), there is a functional
\[
        T_{u,v}\in S_{(X\widehat\otimes_\pi Y)^*}
\]
such that
\[
        T_{u,v}(u\otimes v)=1,
\]
and the family \(\{T_{u,v}:u\in U,\ v\in V\}\) strongly exposes
\(U\otimes V\) uniformly.

Fix \(v_0\in V\). For each \(u\in U\), define
\[
        f_u(x):=T_{u,v_0}(x\otimes v_0),
        \qquad x\in X.
\]
Then \(\|f_u\|\le 1\), while
\[
        f_u(u)=T_{u,v_0}(u\otimes v_0)=1.
\]
Therefore \(f_u\in S_{X^*}\) and \(f_u(u)=1\).

We show that the family \(\{f_u:u\in U\}\) strongly exposes \(U\)
uniformly. Let \(\varepsilon>0\). By the uniform strong exposure of
\(U\otimes V\), there is \(\delta>0\) such that whenever
\(z\in B_{X\widehat\otimes_\pi Y}\), \(u\in U\), and \(v\in V\) satisfy
\[
        T_{u,v}(z)>1-\delta,
\]
then
\[
        \|z-u\otimes v\|_\pi<\varepsilon.
\]
Now let \(u\in U\) and \(x\in B_X\) satisfy
\[
        f_u(x)>1-\delta.
\]
Then \(x\otimes v_0\in B_{X\widehat\otimes_\pi Y}\), and
\[
        T_{u,v_0}(x\otimes v_0)=f_u(x)>1-\delta.
\]
Hence
\[
        \|x\otimes v_0-u\otimes v_0\|_\pi<\varepsilon.
\]
Since \(\|v_0\|=1\), we have
\[
        \|x\otimes v_0-u\otimes v_0\|_\pi
        =
        \|(x-u)\otimes v_0\|_\pi
        =
        \|x-u\|.
\]
Thus \(\|x-u\|<\varepsilon\). Therefore \(U\) is a uniformly strongly
exposed subset of \(B_X\).

The proof for \(V\) is symmetric, and is omitted. Consequently \(U\) and \(V\) are uniformly strongly exposed subsets of
\(B_X\) and \(B_Y\), respectively.
\end{proof}

By the standard description of the
unit ball of the projective tensor product:
\[
        B_{X\pten Y}=\overline{\operatorname{conv}}(B_X\otimes B_Y),
\]
one immediately obtains the following consequence.

\begin{corollary}\label{cor: USE sets in projective tensor products}
Let \(X\) and \(Y\) be Banach spaces. Suppose that
\[
        B_X=\overline{\operatorname{conv}}(U)
        \qquad\text{and}\qquad
        B_Y=\overline{\operatorname{conv}}(V),
\]
where \(U\subset S_X\) and \(V\subset S_Y\) are sets of uniformly strongly
exposed points of \(B_X\) and \(B_Y\), respectively. Then
\[
        B_{X\widehat\otimes_\pi Y}
        =\overline{\operatorname{conv}}(U\otimes V),
\]
and \(U\otimes V\) is a set of uniformly
strongly exposed points of \(B_{X\widehat\otimes_\pi Y}\).
\end{corollary}

\begin{theorem}\label{thm: UBCP in L(X,Y*)}
    Let $X$ and $Y$ be separable Banach spaces such that $B_X$ and $B_Y$ are the closed convex hulls of sets of uniformly strongly exposed points. Then $\mathcal{L}(X,Y^*)$ and $\mathcal{L}(Y,X^*)$ both have the UBCP, and $\bc (\mathcal{L}(X,Y^*))=\bc(\mathcal{L}(Y,X^*))=1$. 
\end{theorem}
\begin{proof}
    Let $X$ and $Y$ be separable Banach spaces such that $B_X$ and $B_Y$ are the closed convex hulls of sets of uniformly strongly exposed points. Then, by Corollary~\ref{cor: USE sets in projective tensor products}, the same is true for $ B_{X\pten Y}$. Since $\mathcal{L}(X,Y^*)\cong(X\pten Y)^*\cong\mathcal{L}(Y,X^*)$, by Theorem~\ref{thm: UBCP in dual}, we can conclude that $\mathcal{L}(X,Y^*)$ and $\mathcal{L}(Y,X^*)$ both have the UBCP, and $\bc (\mathcal{L}(X,Y^*))=\bc(\mathcal{L}(Y,X^*))=1$.
\end{proof}

From \cite[Corollary~3.3]{LLLZ2022}, we know that $\mathcal{L}(\ell_1,\ell_\infty)$ has the BCP. With this new theorem at hand, we can say that this space even has the UBCP.

\begin{example}
The space $\mathcal{L}(\ell_1,\ell_\infty)$ has the UBCP and $\bc (\mathcal{L}(\ell_1,\ell_\infty))=1$, because $B_{\ell_1}$ satisfies the assumptions of Theorem~\ref{thm: UBCP in L(X,Y*)}.
\end{example}

In \cite[Corollary~3.1]{LLLZ2022}, it is shown that $\mathcal{L}(\ell_p)$ has the UBCP whenever $1\leq p<\infty$. 
In \cite[Theorem~3.3]{BLS2025}, the authors prove that $\mathcal{L}(L_p[0,1])$ has the UBCP whenever $3/2<p<3$ and they wonder if this holds for all $1<p<\infty$ \cite[Question 3]{BLS2025}. With the results established so far, we can answer this question positively and give an alternative way to simultaneously prove \cite[Theorem~3.3]{BLS2025} and \cite[Corollary~3.1]{LLLZ2022} for $1<p<\infty$.

Recall that if a Banach space $X$ is uniformly convex, then \(S_X\) is a uniformly strongly exposed subset of \(B_X\). Examples of uniformly convex spaces include $\ell_p$ and  $L_p[0,1]$ whenever $1<p<\infty$. 

\begin{corollary}
Let $1<p<\infty$ and let $X_p$ be either the space $\ell_p$ or $L_p[0,1]$. Then the space $\mathcal{L}(X_p)$ has the UBCP and $\bc(\mathcal{L}(X_p))=1$.
\end{corollary}
\begin{proof}
    Let $1<p<\infty$. Then $X_p$ is a separable Banach space which is also uniformly convex, hence its unit ball is the closed convex hull of a set of uniformly strongly exposed points. Let $1<q<\infty$ be the conjugate exponent, so that $1/p+1/q=1$. Then $X_q$ is also a separable Banach space with the same property. Hence, by Theorem~\ref{thm: UBCP in L(X,Y*)}, $\mathcal{L}(X_p, (X_q)^*)=\mathcal{L}(X_p, X_p)$ has the UBCP and $\bc(\mathcal{L}(X_p))=1$.
\end{proof}

\section{Applications to Lipschitz spaces}\label{sec: Lipschitz spaces}
For infinite metric spaces, ball-covering properties of
\(\operatorname{Lip}_0(M)\) appear to be poorly understood. There are,
however, two useful sources of examples. On the positive side, if \(M\) is
separable and complete and \(\mathcal F(M)\) has the Radon--Nikod\'ym
property, then \(B_{\mathcal F(M)}\) is the closed convex hull of its denting
points, and \cite[Lemma~14]{GLM2019} yields \(bc(\operatorname{Lip}_0(M))=1\).
In particular, this applies to separable complete purely \(1\)-unrectifiable
metric spaces by \cite[Theorem C]{AGPP2022}. On the negative side, if \(M\)
is a length metric space, then \(\operatorname{Lip}_0(M)\) has the Daugavet
property and fails the \((-1)\)-BCP; see \cite[Proposition~5.2]{CLL2023}. In this section, we will aim to produce many new examples of Lipschitz spaces with the ball-covering properties. 

Let $(M,d)$ be a complete metric space with a distinguished point
$0\in M$. The Lipschitz space $\Lip(M)$ is defined as
the Banach space of all Lipschitz functions $f\colon M\to \mathbb{R}$ such
that $f(0)=0$, endowed with the norm given by the best Lipschitz constant
\[
\|f\|=\sup\left\{
\frac{|f(x)-f(y)|}{d(x,y)}\colon x\neq y\in M
\right\}.
\]
For each $x\in M$, the evaluation functional
$\delta(x)\colon f\mapsto f(x)$ belongs to the dual $\Lip(M)^*$.
The \emph{Lipschitz-free space over $M$} is defined as the closed linear span
\[
\mathcal{F}(M):=\overline{\operatorname{span}}\{\delta(x)\colon x\in M\}.
\]
It is well known that $\mathcal{F}(M)$ is an isometric
predual of $\Lip(M)$, see, e.g., \cite[Theorem~3.3]{W2018}. The most important elements of $\mathcal{F}(M)$ are the so-called
\emph{molecules}, of the form
\[
m_{xy}:=\frac{\delta(x)-\delta(y)}{d(x,y)}
\]
for $x\neq y\in M$. The set of molecules in $\mathcal{F}(M)$ will be denoted by
$\Mol M$. Molecules have norm $1$, and it follows easily from the
Hahn--Banach separation theorem that
\[
B_{\mathcal{F}(M)}
=
\cconv(\Mol M)
\]
see, e.g., \cite[Proposition 3.29]{W2018}.
\subsection{BCP in Lipschitz spaces}

Since $\Lip(M)$ is a dual Banach space, for $\Lip(M)$ to have the BCP it is necessary that the predual $\Free(M)$ is separable (see Lemma~\ref{lem:basic-implications}). The latter is equivalent to requiring that the metric space $M$ is separable.

\begin{proposition}\label{prop: BCP in Lipschitz}
   Let $M$ be a separable complete metric space. If $B_{\Free(M)}$ is the closed convex hull of its denting points, then $\bc(\Lip(M))=1$. 
\end{proposition}
\begin{proof}
    By \cite[Theorem~5.1]{LLMRZ2024}, $B_{\Free(M)}$ is an SCD set if and only if $B_{\Free(M)}$ is the closed convex hull of its denting points, hence we can apply Theorem~\ref{thm: BCP in dual} directly.
\end{proof}

Our next aim is to show that there exists a separable metric space such that $\Lip(M)$ has the BCP, but it fails the UBCP. A natural candidate is the unit circle $\mathbb{T}$, because 
$\mathcal{F}(\mathbb{T})$ fails the Radon--Nikod\'ym property, but its unit ball is an SCD set as it is the closed convex hull of its strongly exposed points \cite[Theorem~2.1]{CGLMRZ2021}. 

\begin{theorem}\label{thm:lip-torus-no-ubcp}
Let \(\mathbb T:=\{e^{it}:t\in\mathbb R\}\) be endowed with the chordal metric
\(d(e^{it},e^{is})=|e^{it}-e^{is}|\) and with base point \(1\).
Then \(\Lip(\mathbb T)\) has the BCP, but it fails the UBCP.
\end{theorem}

Theorem~\ref{thm: BCP in dual} immediately gives us that 
\(\bc(\Lip(\mathbb T))=1\), hence \(\Lip(\mathbb T)\) has the BCP. 
It remains to show that \(\Lip(\mathbb T)\) fails the UBCP.

In what follows, \(\mathbb T\) is endowed with the chordal metric and has base point \(1=e^{i0}\). For distinct \(p,q\in\mathbb T\), we write \(m_{p,q}\) for the corresponding molecule. Thus, for \(f\in\Lip(\mathbb T)\),
\[
f(m_{p,q})=\frac{f(p)-f(q)}{d(p,q)}.
\]
For \(0<a<1\), set
\[
    A_a:=\{e^{it}:0\le t\le a\}\subset\mathbb T,
    \qquad
    \kappa(a):=\frac{2\sin(a/2)}{a}.
\]
Then \(\kappa(a)\to1\) as \(a\downarrow0\), and for every \(0\le s<t\le a\),
\[
    \kappa(a)(t-s)
    \le
    d(e^{is},e^{it})
    \le
    t-s.
\]

\begin{lemma}\label{lem:torus-bump}
Let \(0<a<1\) and let \(\tau>0\). Then there exists
\(\rho=\rho(a,\tau)>0\) with the following property: for every measurable
set \(E\subset[0,a]\) with \(|E|<\rho\), there exists
\(f\in S_{\Lip(\mathbb T)}\) such that:
\begin{enumerate}
    \item[(i)] \(f=0\) on \(\mathbb T\setminus A_a\);
    \item[(ii)] the function \(t\mapsto f(e^{it})\) is constant on every connected component of \(E\);
    \item[(iii)] whenever one of the points \(p,q\) belongs to \(A_a\) and the other belongs to
    \(\mathbb T\setminus A_a\), one has \( |f(m_{p,q})|<\tau \).
\end{enumerate}
\end{lemma}

\begin{proof}
Choose a closed interval \(I=[b,b+\ell]\subset(0,a)\) so short that
\[
    \ell<2\tau \kappa(a)\dist(I,\{0,a\}).
\]
Choose \(0<\rho<\ell\), and let \(E\subset[0,a]\) be measurable with
\(|E|<\rho\). Put
\[
    S(t):=|[b,t]\setminus E|,\qquad t\in I.
\]
Define \(g:\mathbb T\to\mathbb R\) by
\[
g(z):=
\begin{cases}
    \min\{S(t),|I\setminus E|-S(t)\},
        & \text{if } z=e^{it}\text{ for some }t\in I,\\
    0,
        & \text{if } z\in \mathbb T\setminus\{e^{it}:t\in I\}.
\end{cases}
\]
This is well-defined since \(I\subset(0,a)\) and \(a<1<2\pi\). 

Let \(d_{\mathrm{arc}}\) denote the intrinsic arc metric on \(\mathbb T\).
Since
\[
 |S(t)-S(s)|\le |t-s| \qquad (s,t\in I),
\]
and the map
\[
u\longmapsto \min\{u,|I\setminus E|-u\}
\]
is \(1\)-Lipschitz, the function \(t\mapsto g(e^{it})\) is
\(1\)-Lipschitz on \(I\). Moreover,
\[
g(e^{ib})=g(e^{i(b+\ell)})=0,
\]
so its zero extension is \(1\)-Lipschitz with respect to
\(d_{\mathrm{arc}}\). Since
\[
d_{\mathrm{arc}}(z,w)\le \frac{\pi}{2}|z-w|
\qquad (z,w\in\mathbb T),
\]
and \(g(1)=0\), it follows that \(g\in\operatorname{Lip}_0(\mathbb T)\). Thus \(g=0\) on \(\mathbb T\setminus A_a\), and \(g\ne0\).
Since \(S\) is constant on every connected subinterval of \(E\cap I\), the function
\(t\mapsto g(e^{it})\) is constant on every connected component of \(E\).

Put \(f:=g/\|g\|_L\). Then \(f\in S_{\Lip(\mathbb T)}\), \(f=0\) on
\(\mathbb T\setminus A_a\), and \(t\mapsto f(e^{it})\) is constant on every connected
component of \(E\). Moreover, for \(t\in I\),
\[
    |f(e^{it})|
    \le
    \min\{d(e^{it},e^{ib}),d(e^{it},e^{i(b+\ell)})\}
    \le
    \frac{\ell}{2}.
\]

It remains to check the boundary estimate. Let \(p=e^{ix}\in A_a\) and
\(q\in\mathbb T\setminus A_a\). If \(x\notin I\), then \(f(p)=f(q)=0\). If \(x\in I\), then
\[
    d(p,q)\ge 2\sin\frac{\dist(I,\{0,a\})}{2}
    \ge \kappa(a)\dist(I,\{0,a\}),
\]
and hence
\[
    |f(m_{p,q})|
    =
    \frac{|f(p)|}{d(p,q)}
    \le
    \frac{\ell}{2\kappa(a)\dist(I,\{0,a\})}
    <
    \tau.
\]
The case \(q\in A_a\) and \(p\in\mathbb T\setminus A_a\) is symmetric.
\end{proof}

\begin{lemma}\label{lem:torus-slope}
Let \(a,\varepsilon\in(0,1)\), and let \(h\in S_{\Lip(\mathbb T)}\). Assume that for some interval
\([s,t]\subset[0,a]\)
one has
\[
    |h(m_{e^{is},e^{it}})|>1-\varepsilon.
\]
Then, for every \(\delta>0\), there exists a subinterval
\([u,v]\subset[s,t]\)
with \(0<v-u<\delta\) such that
\[
    |h(m_{e^{iu},e^{iv}})|
    \ge
    (1-\varepsilon)\kappa(a).
\]
\end{lemma}

\begin{proof}
Assume the conclusion fails. Then there exists \(\delta_0>0\) such that, whenever
\(u,v\in[s,t]\) and \(0<v-u<\delta_0\),
\[
    |h(m_{e^{iu},e^{iv}})|
    <
    (1-\varepsilon)\kappa(a).
\]
Equivalently,
\[
    |h(e^{iv})-h(e^{iu})|
    <
    (1-\varepsilon)\kappa(a)d(e^{iu},e^{iv})
    \le
    (1-\varepsilon)\kappa(a)(v-u).
\]
Choose a partition \(s=t_0<t_1<\cdots<t_N=t\) with
\(t_i-t_{i-1}<\delta_0\) for all \(i\). Then
\[
\begin{aligned}
    |h(e^{it})-h(e^{is})|
    &\le
    \sum_{i=1}^N
    |h(e^{it_i})-h(e^{it_{i-1}})|  \\
    &<
    (1-\varepsilon)\kappa(a)(t-s).
\end{aligned}
\]
So we have
\[
    |h(m_{e^{is},e^{it}})|
    <
    \frac{(1-\varepsilon)\kappa(a)(t-s)}
         {d(e^{is},e^{it})}
    \le
    1-\varepsilon,
\]
which contradicts the assumption.
\end{proof}

Finally, we are ready to prove the main result in this subsection.

\begin{proof}[Proof of Theorem~\ref{thm:lip-torus-no-ubcp}]
Let \(X:=\Lip(\mathbb T)\). Suppose, towards a contradiction, that \(X\) has the UBCP. By Lemma~\ref{lem:ubcp-normal-form}, we may assume that
there exist \(0<\alpha<r\) and a sequence \((g_n)\subset S_X\) such that
\[
    S_X
    \subset
    \bigcup_{n=1}^{\infty}B(rg_n,r-\alpha).
\]

Choose \(0<\varepsilon<\alpha/r\), \(0<\tau<\alpha-r\varepsilon\), and \(0<a<1\) such that
\[
    \kappa(a)>\frac{1-\alpha/r}{1-\varepsilon}.
\]
Define
\[
D:=\bigl\{f\in S_X:\; f=0\text{ on }\mathbb T\setminus A_a,
 |f(m_{p,q})|<\tau
\text{ whenever exactly one of }p,q\text{ belongs to }A_a
\bigr\}.
\]
The set \(D\) is non-empty by Lemma~\ref{lem:torus-bump} applied with \(E=\varnothing\). Let
\[
    J:=\{n\in\mathbb N:
    B(rg_n,r-\alpha)\cap D\ne\varnothing\}.
\]
If \(J=\varnothing\), then any element of \(D\) is not covered, a contradiction. Hence assume
\(J\ne\varnothing\).

We first show that, for \(n\in J\), every pair producing a slope larger
than \(1-\varepsilon\) must lie inside the arc \(A_a\).
For each \(n\in J\), choose \(F_n\in D\) such that
\[
    \|rg_n-F_n\|_L<r-\alpha.
\]
Fix \(n\in J\). If exactly one of the points \(p,q\) lies in \(A_a\), then
\[
\begin{aligned}
    |g_n(m_{p,q})|
    &\le
    \frac1r\left(|(rg_n-F_n)(m_{p,q})|+|F_n(m_{p,q})|\right)\\
    &<
    \frac1r(r-\alpha+\tau)
    <
    1-\varepsilon.
\end{aligned}
\]
If both \(p\) and \(q\) lie outside \(A_a\), then \(F_n(p)=F_n(q)=0\), and hence again
\[
    |g_n(m_{p,q})|
    =
    \frac1r|(rg_n-F_n)(m_{p,q})|
    <
    \frac1r(r-\alpha)
    <
    1-\varepsilon.
\]
Since \(\|g_n\|_L=1\), choose \(p_n,q_n\in\mathbb T\), \(p_n\ne q_n\), such that
\[
    |g_n(m_{p_n,q_n})|>1-\varepsilon.
\]
Then both \(p_n\) and \(q_n\) belong to \(A_a\). Hence, after interchanging them if necessary, we may write
\(p_n=e^{ix_n}\), \(q_n=e^{iy_n}\), where \(0\le x_n<y_n\le a\).

Let \(\rho=\rho(a,\tau)>0\) be given by Lemma~\ref{lem:torus-bump}. Choose positive numbers
\((\delta_n)_{n\in J}\) such that \(\sum_{n\in J}\delta_n<\rho\). By
Lemma~\ref{lem:torus-slope}, for each \(n\in J\) there exist
\(u_n,v_n\in[x_n,y_n]\) such that \(0<v_n-u_n<\delta_n\) and
\[
    |g_n(m_{e^{iu_n},e^{iv_n}})|
    \ge
    (1-\varepsilon)\kappa(a).
\]

Put
\[
    E:=\bigcup_{n\in J}[u_n,v_n]\subset[0,a].
\]
Then \(E\) is measurable and
\[
    |E|\le\sum_{n\in J}(v_n-u_n)<\sum_{n\in J}\delta_n<\rho.
\]
Apply Lemma~\ref{lem:torus-bump} to this set \(E\). We obtain \(f\in S_X\) such that \(f\in D\) and
\(t\mapsto f(e^{it})\) is constant on every connected component of \(E\).
For every \(n\in J\), the interval \([u_n,v_n]\) is contained in a connected component of \(E\). Therefore
\(f(e^{iu_n})=f(e^{iv_n})\), so \(f(m_{e^{iu_n},e^{iv_n}})=0\). Consequently,
\[
\begin{aligned}
    \|rg_n-f\|_L
    &\ge
    |(rg_n-f)(m_{e^{iu_n},e^{iv_n}})|\\
    &=
    r|g_n(m_{e^{iu_n},e^{iv_n}})|\\
    &\ge
    r(1-\varepsilon)\kappa(a)
    >
    r-\alpha.
\end{aligned}
\]
Thus \(f\notin B(rg_n,r-\alpha)\) for every \(n\in J\).

If \(n\notin J\), then \(B(rg_n,r-\alpha)\cap D=\varnothing\). Since \(f\in D\), we also have
\(f\notin B(rg_n,r-\alpha)\) for every \(n\notin J\). Consequently \(f\in S_X\) is not contained
in any ball from the covering, a contradiction.
\end{proof}

We do not know whether \(\Lip(\mathbb T)\) also fails the SBCP. This leads us to the following question. 
\begin{question}
Does \(\Lip(\mathbb T)\) have the SBCP?
\end{question}

\subsection{UBCP in Lipschitz spaces}
Finally, we will give a class of separable metric spaces $M$ for which $\Lip(M)$ will have the UBCP. This class will include separable ultrametric spaces and Hölder metric spaces. For this, we first introduce some notation.

Let $M$ be a metric space. Given $x,y,z\in M$, the \emph{Gromov product} (see, e.g. \cite{BH1999}) of $x$ and $y$ at $z$ is defined as
\[
(x,y)_z:=\frac{1}{2}(d(x,z)+d(z,y)-d(x,y))\geq 0.
\]

\begin{definition}[{\cite[Definition~3.5]{CCGLMRZ2019}, \cite[Definition~3.1]{CM2019}}]\label{def: Gromov concave}
    Let $M$ be a pointed metric space.
    \begin{itemize}
        \item[(a)] The metric space $M$ is said to be \emph{Gromov concave} if for every $x,y\in M$, $x\neq y$, there is $\varepsilon_{x,y}>0$ such that
        \[
(x,y)_z>\varepsilon_{x,y}\min\{d(x,z), d(y,z)\}
               \]
               for every $z\in M\setminus\{x,y\}$.
        \item[(b)] The metric space $M$ is said to be \emph{uniformly Gromov concave} if there is $\varepsilon_{0}>0$ such that
        \[
(x,y)_z>\varepsilon_{0}\min\{d(x,z), d(y,z)\}
               \]
               for every distinct $x,y,z\in M$.
    \end{itemize}
\end{definition}
If \(M\) is separable, complete, and Gromov concave, then
\(\operatorname{Lip}_0(M)\) has the BCP, because every molecule is a
strongly exposed point of \(B_{\mathcal F(M)}\)
\cite[Theorem 5.4]{GLPRZ2018}, hence \(B_{\mathcal F(M)}\) is an SCD set and we can apply Theorem~\ref{thm: BCP in dual}. Examples of uniformly Gromov concave metric spaces include, e.g. ultrametric spaces, because in this case 
\[
(x,y)_z\geq\frac12\min\{d(x,z),d(y,z)\}\qquad\text{for all distinct $x,y,z\in M$},
\]
and Hölder metric spaces \cite[Proposition~3.8]{CM2019}. Note that the metric space $\mathbb{T}$ with the chordal metric is an example of a Gromov concave metric space, which is not uniformly Gromov concave \cite[Theorem~2.1]{CGLMRZ2021}. We will now prove that uniformly Gromov concave metric spaces are sufficient for the UBCP of $\Lip(M)$. 

\begin{theorem}\label{thm: UBCP in Lip}
    If $M$ is a separable complete uniformly Gromov concave metric space, then $\Lip(M)$ has the UBCP and $\bc(\Lip(M))=1$.
\end{theorem}
\begin{proof}
    Since $M$ is a uniformly Gromov concave metric space, the set $\Mol{M}$ is a set of uniformly strongly exposed points of $B_{\mathcal{F}(M)}$ \cite[Remark~3.2]{CM2019}. Therefore, $B_{\mathcal{F}(M)}$ is the closed convex hull of a uniformly strongly exposed set, and we can apply Theorem~\ref{thm: UBCP in dual}.
    \end{proof}

\section*{Acknowledgements}
The authors are grateful to Miguel Mart\'in, Yo\"el Perreau, and Abraham Rueda Zoca for answering some inquiries on the topic of the paper.

\bibliographystyle{amsplain}
\bibliography{references}

\providecommand{\bysame}{\leavevmode\hbox to3em{\hrulefill}\thinspace}
\providecommand{\MR}{\relax\ifhmode\unskip\space\fi MR }
\providecommand{\MRhref}[2]{%
  \href{http://www.ams.org/mathscinet-getitem?mr=#1}{#2}
}
\providecommand{\href}[2]{#2}
\begin{thebibliography}{10}

\bibitem{AK2006}
Fernando Albiac and Nigel~J. Kalton, \emph{Topics in {Banach} space theory}, Grad. Texts Math., vol. 233, Berlin: Springer, 2006.

\bibitem{AGPP2022}
Ram{\'o}n~J. Aliaga, Chris Gartland, Colin Petitjean, and Anton{\'{\i}}n Proch{\'a}zka, \emph{Purely 1-unrectifiable metric spaces and locally flat {Lipschitz} functions}, Trans. Am. Math. Soc. \textbf{375} (2022), no.~5, 3529--3567.

\bibitem{AKMMS2010}
Antonio Avil{\'e}s, Vladimir Kadets, Miguel Mart{\'{\i}}n, Javier Mer{\'{\i}}, and Varvara Shepelska, \emph{Slicely countably determined {Banach} spaces}, Trans. Am. Math. Soc. \textbf{362} (2010), no.~9, 4871--4900.

\bibitem{BLSpreprint}
Qiyao Bao, Rui Liu, and Jie Shen, \emph{Approximation properties and quantitative estimation for uniform ball-covering property of operator spaces}, Preprint, {arXiv}:2507.02261 [math.{FA}] (2025).

\bibitem{BLS2025}
Qiyao Bao, Rui Liu, and Jie Shen, \emph{The ball-covering property of non-commutative spaces of operators on {Banach} spaces}, Banach J. Math. Anal. \textbf{19} (2025), no.~3, 20, Id/No 34.

\bibitem{BH1999}
Martin~R. Bridson and Andr{\'e} Haefliger, \emph{Metric spaces of non-positive curvature}, Grundlehren Math. Wiss., vol. 319, Berlin: Springer, 1999.

\bibitem{CCGLMRZ2019}
Bernardo Cascales, Rafael Chiclana, Luis~C. Garc{\'{\i}}a-Lirola, Miguel Mart{\'{\i}}n, and Abraham Rueda~Zoca, \emph{On strongly norm attaining {Lipschitz} maps}, J. Funct. Anal. \textbf{277} (2019), no.~6, 1677--1717.

\bibitem{Cheng2006}
Lixin Cheng, \emph{Ball-covering property of {Banach} spaces}, Isr. J. Math. \textbf{156} (2006), 111--123.

\bibitem{Cheng2011}
\bysame, \emph{Erratum to: ``{Ball}-covering property of {Banach} spaces''}, Isr. J. Math. \textbf{184} (2011), 505--507.

\bibitem{CCL2008}
Lixin Cheng, Qingjin Cheng, and Xiaoyan Liu, \emph{Ball-covering property of {Banach} spaces that is not preserved under linear isomorphisms}, Sci. China, Ser. A \textbf{51} (2008), no.~1, 143--147.

\bibitem{CKWZ2010}
Lixin Cheng, Vladimir Kadets, Bo~Wang, and Wen Zhang, \emph{A note on ball-covering property of {Banach} spaces}, J. Math. Anal. Appl. \textbf{371} (2010), no.~1, 249--253.

\bibitem{CGLMRZ2021}
Rafael Chiclana, Luis~C. Garc{\'{\i}}a-Lirola, Miguel Mart{\'{\i}}n, and Abraham Rueda~Zoca, \emph{Examples and applications of the density of strongly norm attaining {Lipschitz} maps}, Rev. Mat. Iberoam. \textbf{37} (2021), no.~5, 1917--1951.

\bibitem{CM2019}
Rafael Chiclana and Miguel Mart{\'{\i}}n, \emph{The {Bishop}-{Phelps}-{Bollob{\'a}s} property for {Lipschitz} maps}, Nonlinear Anal., Theory Methods Appl., Ser. A, Theory Methods \textbf{188} (2019), 158--178.

\bibitem{CLL2023}
Stefano Ciaci, Johann Langemets, and Aleksei Lissitsin, \emph{A characterization of {Banach} spaces containing {{\(\ell_1(\kappa)\)}} via ball-covering properties}, Isr. J. Math. \textbf{253} (2023), no.~1, 359--379.

\bibitem{FHHMZ2011}
Mari{\'a}n Fabian, Petr Habala, Petr H{\'a}jek, Vicente Montesinos, and V{\'a}clav Zizler, \emph{Banach space theory. {The} basis for linear and nonlinear analysis}, CMS Books Math./Ouvrages Math. SMC, Berlin: Springer, 2011.

\bibitem{FZ2009}
Vladimir~P. Fonf and Clemente Zanco, \emph{Covering spheres of {Banach} spaces by balls}, Math. Ann. \textbf{344} (2009), no.~4, 939--945.

\bibitem{GLPRZ2018}
Luis Garc{\'{\i}}a-Lirola, Anton{\'{\i}}n Proch{\'a}zka, and Abraham Rueda~Zoca, \emph{A characterisation of the {Daugavet} property in spaces of {Lipschitz} functions}, J. Math. Anal. Appl. \textbf{464} (2018), no.~1, 473--492.

\bibitem{GLM2019}
A.~J. Guirao, A.~Lissitsin, and V.~Montesinos, \emph{Some remarks on the ball-covering property}, J. Math. Anal. Appl. \textbf{479} (2019), no.~1, 608--620.

\bibitem{HKL2024}
Jinghao Huang, Karimbergen Kudaybergenov, and Rui Liu, \emph{The ball-covering property of noncommutative symmetric spaces}, Stud. Math. \textbf{277} (2024), no.~2, 169--190.

\bibitem{KMMW2013}
Vladimir Kadets, Miguel Mart{\'{\i}}n, Javier Mer{\'{\i}}, and Dirk Werner, \emph{Lushness, numerical index 1 and the {Daugavet} property in rearrangement invariant spaces}, Can. J. Math. \textbf{65} (2013), no.~2, 331--348.

\bibitem{LLMRZ2024}
Johann Langemets, Marcus L{\~o}o, Miguel Mart{\'{\i}}n, and Abraham Rueda~Zoca, \emph{Slicely countably determined points in {Banach} spaces}, J. Math. Anal. Appl. \textbf{537} (2024), no.~1, 36, Id/No 128248.

\bibitem{L1963}
Joram Lindenstrauss, \emph{On operators which attain their norm}, Isr. J. Math. \textbf{1} (1963), 139--148.

\bibitem{LLLZ2022}
Minzeng Liu, Rui Liu, Jimeng Lu, and Bentuo Zheng, \emph{Ball covering property from commutative function spaces to non-commutative spaces of operators}, J. Funct. Anal. \textbf{283} (2022), no.~1, 15, Id/No 109502.

\bibitem{LP2026}
Marcus L{\~o}o and Yo{\"e}l Perreau, \emph{A {Banach} space with an unconditional basis which is not slicely countably determined}, Preprint, {arXiv}:2603.12947 [math.{FA}] (2026).

\bibitem{LZ2020}
Zhenghua Luo and Bentuo Zheng, \emph{Stability of the ball-covering property}, Stud. Math. \textbf{250} (2020), no.~1, 19--34.

\bibitem{LuoZheng2021}
\bysame, \emph{The strong and uniform ball-covering properties}, J. Math. Anal. Appl. \textbf{499} (2021), no.~1, 15, Id/No 125034.

\bibitem{PS2000}
Rafael Pay{\'a} and Yousef Saleh, \emph{Norm attaining operators from {{\(L_1 (\mu)\)}} into {{\(L_{\infty} (\nu)\)}}}, Arch. Math. \textbf{75} (2000), no.~5, 380--388.

\bibitem{P1974}
R.~R. Phelps, \emph{Dentability and extreme points in {Banach} spaces}, J. Funct. Anal. \textbf{17} (1974), 78--90.

\bibitem{RS1986}
W.~M. Ruess and C.~P. Stegall, \emph{Exposed and denting points in duals of operator spaces}, Isr. J. Math. \textbf{53} (1986), 163--190.

\bibitem{R2002}
Raymond~A. Ryan, \emph{Introduction to tensor products of {Banach} spaces}, Springer Monogr. Math., London: Springer, 2002.

\bibitem{Sain2025}
Debmalya Sain, \emph{On norm derivatives and the ball-covering property of {Banach} spaces}, J. Math. Anal. Appl. \textbf{541} (2025), no.~1, 9, Id/No 128738.

\bibitem{S1983}
Walter Schachermayer, \emph{Norm attaining operators and renormings of {Banach} spaces}, Isr. J. Math. \textbf{44} (1983), 201--212.

\bibitem{W2018}
Nik Weaver, \emph{Lipschitz algebras}, 2nd ed., Hackensack, NJ: World Scientific, 2018.

\end{thebibliography}

\end{document}